\documentclass[a4paper,11pt]{amsart}
\usepackage[active]{srcltx}
\usepackage{latexsym,amssymb}
\usepackage{stmaryrd}
\usepackage{enumerate}
\usepackage{stackrel}
\usepackage[T1]{fontenc}
\usepackage[latin1]{inputenc}
\usepackage{xcolor}
\usepackage{hyperref}
\usepackage{pgf,tikz}
\usetikzlibrary{bending,positioning,arrows,automata,snakes,shapes}
\usepackage{xypic}
\usepackage{subfig}
\usepackage{mathtools}

\newcommand{\J}{\ensuremath{\mathcal{J}}}

\newcommand{\R}{\ensuremath{\mathcal{R}}}
\renewcommand{\L}{\ensuremath{\mathcal{L}}}
\renewcommand{\H}{\ensuremath{\mathcal{H}}}
\newcommand{\freeprod}{\mathop{\mbox{\Large$*$}}}

\newdir^{((}{{}*!/-5pt/\dir^{(}}
\renewcommand{\geq}{\geqslant}
\renewcommand{\leq}{\leqslant}
\renewcommand{\ge}{\geqslant}

\let\op=\llbracket
\let\cl=\rrbracket
\def\pv#1{\ensuremath{\mathsf{#1}}}

\def\Om#1#2{\ensuremath{\overline{\Omega}_{#1}{\pv{#2}}}}

\def\loc1{\ell \pv{I}}
\def\krho{{\mathrm{KR}}}

\let\cal=\mathcal
\def\Cl#1{\ensuremath{\cal{#1}}}
\newcommand\malcev{%
\mathbin{\hbox{$\bigcirc$\kern-9pt\raise1.2pt\hbox{\scriptsize$m$}\,}}}
\newcommand\smalcev{%
\mathbin{\hbox{$\bigcirc$\kern-8pt\raise1.2pt\hbox{\tiny$m$}\,}}}

\newtheorem{Thm}{Theorem}[section]
\newtheorem{Prop}[Thm]{Proposition}
\newtheorem{Lemma}[Thm]{Lemma}

\newtheorem{Problem}[Thm]{Problem}

\theoremstyle{definition}
\newtheorem{Def}[Thm]{Definition}
\newtheorem{Remark}[Thm]{Remark}
\newtheorem{Examp}[Thm]{Example}

\numberwithin{equation}{section}

\begin{document}

\title { Equidivisibility and profinite coproduct } %
\thanks{The work of J.~Almeida, was partially supported by CMUP, which
  is funded by FCT (Portugal) by national funds through the project
  UID/MAT/00144/2020. It was developed in part at Masaryk University,
  whose hospitality is gratefully acknowledged, with the support of
  the FCT sabbatical scholarship SFRH/BSAB/142872/2018. %
  The work of A.~Costa was carried out in part at Masaryk University
  and at City College of New York. The hospitality of these
  institutions is gratefully acknowledged. The first visit had the
  support of the research grant 19-12790S of the Grant Agency of the
  Czech Republic, and the second visit had the support of the FCT
  sabbatical scholarship SFRH/BSAB/150401/2019. His work was also
  supported by the Centre for Mathematics of the University of Coimbra
  - UIDB/00324/2020, funded by the Portuguese Government through
  FCT/MCTES%
}

\author{Jorge Almeida}
\address{CMUP, Departamento de Matem\'atica,
  Faculdade de Ci\^encias, Universidade do Porto, 
  Rua do Campo Alegre 687, 4169-007 Porto, Portugal.}
\email{jalmeida@fc.up.pt}

\author{Alfredo Costa}
\address{University of Coimbra, CMUC, Department of Mathematics,
  Apartado 3008, EC Santa Cruz,
  3001-501 Coimbra, Portugal.}
\email{amgc@mat.uc.pt}

\begin{abstract}
  The aim of this work is to investigate the behavior of
  equidivisibility under coproduct in the category of pro-\pv V
  semigroups, where \pv V is a pseudovariety of finite semigroups.
  Exploring the relationship with the two-sided Karnofsky--Rhodes
  expansion, the notions of KR-cover and strong KR-cover for profinite
  semigroups are introduced. The former is stronger than
  equidivisibility and the latter provides a characterization of
  equidivisible profinite semigroups with an extra mild condition,
  so-called letter super-cancellativity. Furthermore, under the
  assumption that \pv V is closed under two-sided Karnofsky--Rhodes expansion,
  closure of some classes of equidivisible pro-\pv V semigroups under
  (finite) \pv V-coproduct is established.
\end{abstract}


\makeatletter
\@namedef{subjclassname@2020}{%
  \textup{2020} Mathematics Subject Classification}
\makeatother

\maketitle

%\tableofcontents

\section{Introduction}
\label{sec:introduction}

A semigroup is \emph{equidivisible} if any two factorizations of every
element have a common refinement. The class of equidivisible
semigroups was introduced and studied in~\cite{McKnight&Storey:1969}
as a natural common generalization of free semigroups and completely
simple semigroups. More recently, this property has appeared as a
useful tool in profinite semigroup theory, beginning
with~\cite{Almeida&ACosta:2007a,Henckell&Rhodes&Steinberg:2010b},
where it was noted, independently, that for several important
pseudovarieties of finite semigroups (like that of all finite
semigroups, or that of all finite aperiodic semigroups), the
corresponding finitely generated relatively free profinite semigroups
are equidivisible. Other recent papers where the equidivisibility of
relatively free profinite semigroups is applied or deserves some kind
of attention
include~\cite{Almeida&Klima:2020a,Almeida&ACosta&Costa&Zeitoun:2019,Gool&Steinberg:2019}.
A complete characterization of the pseudovarieties for which the
corresponding finitely generated relatively free profinite semigroups
are equidivisible appears in~\cite{Almeida&ACosta:2017}.

As observed in~\cite{McKnight&Storey:1969}, the class of equidivisible
semigroups is closed under taking free products, that is, coproducts
in the category of semigroups. In this paper, we investigate an analog
for profinite semigroups. For that purpose, we introduce \pv
V-coproducts of pro-\pv V semigroups with respect to a pseudovariety
of semigroups \pv V, extending what was done in
\cite{Ribes&Zalesskii:2010} for the pseudovariety of finite groups. We
give simple conditions on \pv V guaranteeing that the free product of
pro-\pv V semigroups embeds naturally in their \pv V-coproduct.

We introduce a restricted form of projectivity. The profinite
semigroups with this property are called KR-covers and turn out to be
equidivisible semigroups. We show that the class of pro-\pv V
KR-covers is closed under \pv V-coproduct when \pv V is closed under
two-sided Karnofsky--Rhodes expansion. This expansion is a two-sided
analog of the so-called Karnofsky--Rhodes expansion
\cite{Elston:1999,Rhodes&Steinberg:2001}, which has recently found new
applications beyond semigroup theory
\cite{Rhodes&Schilling:2019b,Rhodes&Schilling:2019a}.

One of the motivations for searching for new examples of equidivisible
profinite semigroups comes from the fact that several results
in~\cite{Almeida&ACosta&Costa&Zeitoun:2019} were stated for
equidivisible profinite semigroups, frequently with the additional
requirement that they satisfy a certain cancellation property.
Semigroups satisfying this cancellation property are called letter
super-cancellative in~\cite{Almeida&ACosta:2017} and finitely
cancellable in~\cite{Almeida&ACosta&Costa&Zeitoun:2019}. They include
the finitely generated relatively free profinite semigroups that are
equidivisible but not completely simple.
In this paper we provide a characterization of the class of all
finitely generated equidivisible letter super-cancellative profinite
semigroups, involving the notion of strong KR-cover, which we
introduce. We show that the subclass consisting of pro-\pv V
semigroups is also closed under taking finite \pv V-coproducts when
\pv V is closed under two-sided Karnofsky--Rhodes expansion. We also
exhibit an element of the class that is not relatively free (the
existence of such an example was left open
in~\cite{Almeida&ACosta&Costa&Zeitoun:2019}).

\section{Preliminaries}
\label{sec:prelims}

The reader is referred to standard references for general background
on profinite semigroups and pseudovarieties
\cite{Almeida:1994a,Almeida:2003cshort,Rhodes&Steinberg:2009qt}. For
the remainder of the section, we introduce briefly specific notions
and terminology needed in the sequel.

For a semigroup $S$, let $S^I$ be the semigroup which is obtained by
adjoining a new identity element, denoted $I$, even if $S$ is itself a
monoid. The semigroup $S$ is \emph{equidivisible} when, for every
$u,v,x,y\in S$, the equality $uv=xy$ implies the existence of $t\in
S^I$ such that $x=ut$ and $ty=v$, or $xt=u$ and $y=tv$. Equivalently,
any two factorizations of the same element of~$S$ have a common
refinement.

A \emph{pseudovariety (of semigroups)} is a class of finite semigroups
closed under taking homomorphic images, subsemigroups and finite
direct products. For the remainder of the paper, \pv V denotes an
arbitrary pseudovariety.

A \emph{topological semigroup} is a semigroup $S$ endowed with a
topology such that the semigroup multiplication $S\times S\to S$ is
continuous. We then say that a mapping $\varphi:X\to S$, with $X$ a
nonempty set, is a \emph{generating mapping} if $\varphi(X)$ generates
a dense subsemigroup of~$S$. When the generating mapping is
understood, we may simply say that $S$ is \emph{$X$-generated}. When
no topology is mentioned, we consider semigroups endowed with the
discrete topology.

Throughout the paper, when we consider compact spaces, we assume that
they are Hausdorff. A \emph{pro-\pv V semigroup} is a compact
semigroup which is residually \pv V. A \emph{profinite semigroup} is a
pro-\pv S semigroup for the pseudovariety \pv S of all finite
semigroups. Since in a finite semigroup, for every element $s$ the
sequence $(s^{n!})_n$ converges, and the limit is idempotent, the same
holds in an arbitrary profinite semigroup; the limit of the sequence
is denoted $s^\omega$. More generally, the sequence
$(s^{n!+k})_{n\ge|k|}$ converges for every integer $k$ and the limit
is denoted $s^{\omega+k}$.

Pro-\pv V semigroups may be
alternatively described as the inverse limits of inverse systems of
finite semigroups. Here, by an \emph{inverse system}, we mean a family
$(S_i)_{i\in I}$ of compact semigroups, where $I$ is an upper directed
set, together with continuous homomorphisms $\varphi_{i,j}:S_i\to S_j$
whenever $i\ge j$ such that $\varphi_{i,i}$ is the identity mapping
and $\varphi_{j,k}\circ\varphi_{i,j}=\varphi_{i,k}$ whenever $i\ge
j\ge k$. We say that the inverse system is an \emph{inverse quotient
  system} if every mapping $\varphi_i$ is onto. The inverse limit
$\varprojlim_{i\in I}S_i$ of the inverse system $(S_i)_{i\in I}$ is
the subsemigroup of the direct product $\prod_{i\in I}S_i$ consisting
of the elements $(s_i)_{i\in I}$ of such that $\varphi_{i,j}(s_i)=s_j$
whenever $i\ge j$. The restriction to $\varprojlim_{i\in I}S_i$ of the
$j$-component projection is a continuous homomorphism
$\varphi_j:\varprojlim_{i\in I}S_i\to S_j$. In case the $S_i$ are
pro-\pv V semigroups, so is $\varprojlim_{i\in I}S_i$. In the case of
an inverse quotient system, we also say that $\varprojlim_{i\in I}S_i$
is an \emph{inverse quotient limit}; in this case, the mappings
$\varphi_i$ are onto.

Given a nonempty set $X$, there is a \emph{free pro-\pv V semigroup on
  $X$} given by a mapping $\iota:X\to\Om XV$ with the following universal
property: for every mapping $\varphi:X\to S$ into a pro-\pv V
semigroup $S$, there is a unique continuous homomorphism
$\hat{\varphi}:\Om XV\to S$ such that
$\hat{\varphi}\circ\iota=\varphi$. Such semigroups are well known to
exist and may be constructed as inverse limits of $X$-generated
semigroups from~\pv V or, in the case $X$ is finite, as completions of
the free semigroup $X^+$ on $X$ with respect to a natural
pseudometric. From the universal property of \Om XV it follows
immediately that $\iota$ is a generating mapping and, unless \pv V is
the trivial pseudovariety, consisting only of singleton semigroups,
the mapping $\iota$ is injective and we then identify each element
$x\in X$ with its image $\iota(x)$.

A homomorphism $\psi:S\to T$ of finite semigroups is said to be a
\emph{$\pv V$-morphism} if $\varphi^{-1}(e)\in\pv V$ for every
idempotent $e$ of $T$. For pseudovarieties \pv V and \pv W, their
\emph{Mal'cev product} is the pseudovariety $\pv{V\malcev W}$
generated by the class of all finite semigroups $S$ for which there is
a \pv V-morphism $S\to T$ with $T\in\pv W$. For example, it is well
known that $\pv J=\pv{N\malcev Sl}$, where \pv J, \pv N and \pv{Sl}
are, respectively, the pseudovarieties of all finite \Cl J-trivial
semigroups nilpotent semigroups and semilattices.

\section{Coproduct of profinite semigroups}
\label{sec:copr-prof-semigr}

Let \pv V be a pseudovariety of finite semigroups. Given a nonempty family
$(S_i)_{i\in I}$ of pro-\pv V semigroups,
their \emph{\pv V-coproduct} is a pro-\pv V semigroup $S$ together
with a collection of continuous homomorphisms $\varphi_i:S_i\to S$
such that the following universal property holds: for every pro-\pv V
semigroup $T$ and every collection of continuous homomorphisms
$\psi_i:S_i\to T$, there is a unique continuous homomorphism
$\psi:S\to T$ such that Diagram~\ref{eq:definition-of-coproduct}
commutes for every $i\in I$.
\begin{equation}
  \label{eq:definition-of-coproduct}
  \begin{split}
    \xymatrix@C=3mm{
  &S_i \ar[ld]_{\varphi_i} \ar[rd]^{\psi_i} & \\
  S \ar@{-->}[rr]^\psi
  && **[r]{T .} }
  \end{split}
\end{equation}

Note that, by the usual ``abstract nonsense'', if such a pro-\pv V
semigroup~$S$ exists, then it is unique up to isomorphism. It is then
denoted $\coprod_{i\in I}^{\pv V}S_i$. In the case of a finite family
$(S_i)_{i=1,\ldots,n}$ we also write $S_1\amalg^{\pv
  V}\cdots\amalg^{\pv V} S_n$ to denote $\coprod_{i=1}^{\pv V}S_i$.

\subsection{Construction and basic properties of \pv V-coproducts}
\label{sec:constr-basic-props-V-coproducts}

The following is the extension to arbitrary profinite semigroups of
the special case of profinite groups considered
in~\cite[Proposition~9.1.2]{Ribes&Zalesskii:2010}.

\begin{Prop}
  \label{p:existence-coproduct}
  Let $(S_i)_{i\in I}$ be a nonempty family of pro-\pv V semigroups. Then the
  \pv V-coproduct $\coprod_{i\in I}^{\pv V}S_i$ exists.
\end{Prop}

In the proof of Proposition~\ref{p:existence-coproduct}
we use the following alternative characterization of
the \pv V-coproduct, consisting in replacing ``pro-\pv V semigroup $T$'' by ``semigroup $T$ from \pv V''
in the above definition of \pv V-coproduct.

\begin{Prop}\label{p:alternative-definition-of-coproduct}
  Consider a nonempty family
  $(S_i)_{i\in I}$ of pro-\pv V semigroups. The \mbox{pro-\pv V}
  semigroup $S$, together with the collection of continuous
  homomorphisms $\varphi_i:S_i\to S$,
  is the \pv V-coproduct of the family $(S_i)_{i\in I}$
  if and only if for every semigroup $T$ from \pv V and every collection of continuous homomorphisms $\psi_i:S_i\to T$, there is a unique continuous homomorphism
$\psi:S\to T$ such that Diagram~\eqref{eq:definition-of-coproduct}
commutes for every $i\in I$.
\end{Prop}

\begin{proof}
  The ``only if'' part of the proposition is trivial.

  Conversely, let us assume that the restricted version of the
  universal property holds for $S$ whenever $T$ belongs to~\pv V.
  Take now an inverse limit
  $T=\varprojlim_{\lambda\in\Lambda}T_\lambda$ of semigroups
  $T_\lambda$ from \pv V. For each $\lambda\in\Lambda$, denote by
  $\pi_\lambda$ the associated projection $T\to T_\lambda$, and for
  $\mu,\lambda\in\Lambda$ such that $\lambda\leq \mu$, let
  $\pi_{\mu,\lambda}$ be the connecting homomorphism $T_\mu\to
  T_\lambda$. Consider a collection of continuous homomorphisms
  $\psi_i:S_i\to T$. By hypothesis, for each $\lambda\in\Lambda$,
  there is a continuous homomorphism $\psi_\lambda:S\to T_\lambda$
  such that Diagram~\ref{eq:alternative-definition-of-coproduct-1}
  commutes for every $i\in I$.
  \begin{equation}
    \label{eq:alternative-definition-of-coproduct-1}
  \begin{split}
    \xymatrix@C=3mm{
  &S_i \ar[ld]_{\varphi_i} \ar[rd]^{\pi_\lambda\circ\psi_i} & \\
  S \ar@{-->}[rr]^{\psi_\lambda}
  && **[r]{T_\lambda .} }
  \end{split}
\end{equation}
Given $\mu,\lambda\in\Lambda$
such that $\lambda\leq\mu$,
we have
\begin{equation*}
  \pi_{\mu,\lambda}\circ\psi_\mu\circ\varphi_i=\pi_{\mu,\lambda}\circ\pi_\mu\circ\psi_i=\pi_\lambda\circ\psi_i
\end{equation*}
for every $i\in I$. Since $\psi_\lambda$ is the unique
continuous homomorphism for which Diagram~\eqref{eq:alternative-definition-of-coproduct-1} commutes for all $i\in I$,
we deduce that $\pi_{\mu,\lambda}\circ\psi_\mu=\psi_\lambda$.
By the definition of inverse limit, we conclude that
there is a continuous homomorphism $\psi:S\to T$ for which
Diagram~\eqref{eq:alternative-definition-of-coproduct-2}
commutes whenever $\mu,\lambda\in\Lambda$ satisfy $\mu\leq\lambda$.
\begin{equation}\label{eq:alternative-definition-of-coproduct-2}
\begin{split}
  \xymatrix@R=3mm{
    &&T_\mu\ar[dd]\ar[dd]^{\pi_{\mu,\lambda}}\\
    S
    \ar@/^5mm/[rru]^{\psi_\mu}
    \ar@{-->}[r]^{\psi}
    \ar@/_5mm/[rrd]_{\psi_\lambda}
    &T\ar[ru]^{\pi_\mu}\ar[rd]_{\pi_\lambda}&\\
    &&T_\lambda
  }
\end{split}  
\end{equation}
Let $i\in I$. Since for every $\lambda\in\Lambda$ we have
$\pi_\lambda\circ(\psi\circ\varphi_i)=\psi_\lambda\circ\varphi_i
=\pi_\lambda\circ\psi_i$, we conclude that $\psi\circ\varphi_i=\psi_i$,
thus establishing the proposition.
\end{proof}

We may now proceed with the proof of
Proposition~\ref{p:existence-coproduct}. It is well known that
coproducts exist in the category of semigroups (see
\cite[Theorem~I.13.5]{Petrich:1984}). We denote the coproduct, also
known as free product, of a nonempty family $(S_i)_{i\in I}$ of
semigroups by $S=\freeprod_{i\in I}S_i$.

\begin{proof}[Proof of Proposition~\ref{p:existence-coproduct}]
  We start by taking the free product $S=\freeprod_{i\in I}S_i$
  in the category of semigroups and the corresponding natural
  homomorphisms $\tilde{\varphi}_i:S_i\to S$. We consider the family
  \Cl C of all congruences $\theta$ on $S$ such that $S/\theta\in\pv
  V$ and, for each $i\in I$, the congruence
  $(\tilde{\varphi}_i\times\tilde{\varphi}_i)^{-1}(\theta)$ is a
  clopen subset of $S_i\times S_i$. The latter condition expresses
  that the composite of $\tilde{\varphi}_i$ followed by the natural
  quotient mapping $q_\theta:S\to S/\theta$ is continuous. Given
  $\theta,\rho\in\Cl C$, note that $\theta\cap\rho\in\Cl C$ since \pv
  V is closed under taking subsemigroups and finite direct products.
  Hence, the quotients $S/\theta$ with $\theta\in\Cl C$ form an
  inverse system, for which we may consider the inverse limit, which
  we denote $S_{\pv V}$. Thus, $S_{\pv V}$ is a pro-\pv V semigroup.
  Let $\iota:S\to S_{\pv V}$ be the natural homomorphism and let
  $\varphi_i=\iota\circ\tilde{\varphi}_i$ ($i\in I$).

  For every $\theta\in\Cl C$,
  the natural projection $\pi_\theta:S_{\pv V}\to S/\theta$
  satisfies
  $\pi_\theta\circ\iota=q_\theta$,
  and so we have the equalities
  \begin{equation*}
    \pi_\theta\circ\varphi_i=\pi_\theta\circ \iota\circ\tilde\varphi_i
    =q_\theta\circ\tilde\varphi_i,
  \end{equation*}
  for which the reader may wish to refer to
  Diagram~\eqref{eq:existence-coproduct-1} below.  
  Since, by the definition of $\Cl C$,
  the mapping $q_\theta\circ\tilde\varphi_i$
  is continuous for every $\theta\in\Cl C$,
 it follows that each mapping $\varphi_i$ is continuous.

 Suppose that $T\in\pv V$ and $(\psi_i)_{i\in I}$ is a family of
 continuous homomorphisms $\psi_i:S_i\to T$. By the universal property
 of the free product $S$, there exists a unique homomorphism
 $\gamma:S\to T$ such that $\gamma\circ\tilde{\varphi}_i=\psi_i$
 ($i\in I$). Let~$\theta$ be the kernel of $\gamma$. Then
 $\theta\in\Cl C$ and there is a unique homomorphism
 $\beta:S/\theta\to T$ such that $\beta\circ q_\theta=\gamma$, and so
 the non-dashed part of Diagram~\eqref{eq:existence-coproduct-1} is
 commutative.
  \begin{equation}\label{eq:existence-coproduct-1}
  \begin{split}  
   \xymatrix@C=5mm{
     &S_i \ar[ld]_{\tilde{\varphi}_i} \ar[rd]^{\psi_i}
    \ar@/_15mm/[ldd]_{\varphi_i} & \\
    S \ar[rr]^\gamma \ar[d]_\iota \ar[rrd]^(.35){q_\theta} && T \\
    S_{\pv V} \ar[rr]^{\pi_\theta} \ar@{-->}[rru]^(.25)\psi
    && S/{\theta}\ar[u]_{\beta}
  } %
 \end{split}
\end{equation}
Therefore, for the continuous homomorphism $\psi=\beta\circ \pi_\theta$,
we have $\psi\circ \varphi_i=\psi_i$, for every $i\in I$.

Since the restriction to $\iota(S)$ of each $\pi_\theta$ is onto, as
so is $q_\theta$, and the preimages of points under the $\pi_\theta$
form a base of the topology of~$S_{\pv V}$, we see that $\iota(S)$ is
dense in $S_{\pv V}$. On the other hand, $S$ is generated by
$\bigcup_{i\in I}\tilde\varphi_i(S_i)$. Thus,
$\psi$ is completely determined by its restriction to the union
$\bigcup_{i\in I}\varphi_i(S_i)$, and so $\psi$ is the unique
continuous homomorphism from $S_{\pv V}$ to $T$ such that, for every
$i\in I$, the following diagram commutes:
\begin{equation*}
    \xymatrix@C=3mm{
  &S_i \ar[ld]_{\varphi_i} \ar[rd]^{\psi_i} & \\
  S_{\pv V} \ar@{-->}[rr]^\psi
  && **[r]{T.} }
\end{equation*}
Since $T$ is an arbitrary semigroup from~\pv V, it follows from
Proposition~\ref{p:alternative-definition-of-coproduct} that the
profinite semigroup $S_{\pv V}$, together with the family of
continuous homomorphisms $(\varphi_i)_{i\in I}$, provides a \pv
V-coproduct of the family $(S_i)_{i\in I}$.
\end{proof}

Let $k\in I$. The continuous semigroup homomorphism $\varphi_k:S_k\to\coprod^\pv V_{i\in I}S_i$
introduced in
the definition of \pv V-coproduct
is said to be the \emph{natural}
mapping from $S_k$ into $\coprod^\pv V_{i\in I}S_i$.

\begin{Prop}\label{p:the-factors-embed-in-the-free-product}
  If $(S_i)_{i\in I}$ is a nonempty family
  of pro-\pv V semigroups then, for every $k\in I$, the natural mapping
  $\varphi_k:S_k\to\coprod^\pv V_{i\in I}S_i$ is injective, whence an
  embedding of topological semigroups.
\end{Prop}

\begin{proof}
  Let $k\in I$. Since $S_k$ is pro-$\pv V$, in
  Diagram~\ref{eq:definition-of-coproduct} we may take $T=S_k$, choose
  $\psi_k$ as the identity $\mathrm{Id}_{S_k}$ on $S_k$, and if $i\in
  I\setminus\{k\}$, we choose $\psi_i$ as any constant mapping from
  $S_i$ onto an idempotent of $S_k$. We may then consider the
  continuous homomorphism $\psi$ as in
  Diagram~\ref{eq:definition-of-coproduct}, for which we have in
  particular $\psi\circ\varphi_k=\mathrm{Id}_{S_k}$, and so the
  proposition holds.
\end{proof}

In view of Proposition~\ref{p:the-factors-embed-in-the-free-product},
we may henceforth see each $S_i$ as a closed subsemigroup of
$\coprod^\pv V_{i\in I}S_i$, with $\varphi_i$ being the inclusion.

Let \pv{Sl} denote the pseudovariety of all finite semilattices.
The following technical observation will be convenient later on.

 \begin{Lemma}
  \label{l:traces-of-generating-sets}
  Suppose that \pv V contains \pv{Sl}. Let $(S_i)_{i\in I}$ be a
  nonempty family of nontrivial pro-$\pv V$ semigroups. Let $A$ be a
  generating subset of the \pv V-coproduct $\coprod^\pv V_{i\in I}S_i$
  as a topological semigroup. Then $A\cap S_i$ generates $S_i$ as a
  topological semigroup for every $i\in I$.
\end{Lemma}

\begin{proof}
  Let $S=\freeprod_{i\in I}S_i$ and let $S_{\pv V}=\coprod^\pv
  V_{i\in I}S_i$.
  Note that $\bigcup_{i\in I}S_i$ generates $S_{\pv V}$. We should
  show that every element of $S_i$ is the limit of a net of products
  of elements of $A\cap S_i$. With this aim, we first claim that there
  is a continuous homomorphism $\psi_{\pv V}$ from $S_{\pv V}$ to the
  two-element semilattice $\{0,1\}$ such that
  $\psi_{\pv V}^{-1}(1)=S_i$.
  It follows that $S_{\pv V}\setminus S_i=\psi_{\pv V}^{-1}(0)$ is an
  ideal of~$S_{\pv V}$ containing $A\setminus S_i$. Hence, in a net of
  products of elements of~$A$ converging to an element of~$S_i$ all
  products from some point on can only involve factors from $A\cap S_i$.
  
  To establish the claim, we use the universal properties of~$S$
  and~$S_{\pv V}$ to define a homomorphism $\psi$ and a continuous
  homomorphism $\psi_{\pv V}$ by considering the constant
  homomorphisms $\varphi_k$ on the $S_k$, with value 1 for $k=i$ and
  value 0 for $k\in I\setminus\{i\}$. More precisely, we have the
  following commutative diagram for each $k\in I$:
  \begin{displaymath}
    \xymatrix@C=13mm{
      S
      \ar[rd]^\psi
      \ar[d]_\iota
      &
      S_k
      \ar@{_(->}[l]
      \ar[d]^{\varphi_k}
      \\
      S_{\pv V}
      \ar[r]^{\psi_{\pv V}}
      &
      \{0,1\}.
    }
  \end{displaymath}
  Since $\psi_{\pv V}$ is continuous and $\iota(S)$ is dense
  in~$S_{\pv V}$, we obtain the following formula:
  \begin{align*}
    \psi_{\pv V}^{-1}(1)
    &=\overline{\iota(\psi^{-1}(1))}=\iota(S_i)=S_i
  \end{align*}
  which establishes the claim and completes the proof of the lemma.
\end{proof}

Note that the hypothesis that \pv V contains \pv{Sl} may not be
omitted in Lemma \ref{l:traces-of-generating-sets}. For instance, if
$G$ is a cyclic group of order $2$ with generator $a$ and $H$ is a
cyclic group of order $3$ with generator $b$, then $ab^2$ is a
generator of $G\amalg^{\pv{Ab}}H$ (since $(ab^2)^2=b$ and
$(ab^2)^3=a$) but $ab^2\notin G\cup H$.

The next couple of facts (that one should bear in mind, albeit not
needed for the sequel) may also be easily proved with the universal
property of the \pv V-coproduct:
 
\begin{enumerate}[1.]
\item\label{item:coproduct-1} The \pv V-coproduct is associative:
  if the nonempty set $I$ has a partition $I=\biguplus_{j\in J} I_j$,
  and $(S_i)_{i\in I}$ is a family of pro-\pv V semigroups,
  then we have an isomorphism
  \begin{equation*}
    \textstyle\coprod_{i\in I}^{\pv V}S_i\cong\coprod_{j\in J}^{\pv V}(\coprod_{i\in I_j}^{\pv V}S_i)
  \end{equation*}
  of profinite semigroups.
  
\item\label{item:coproduct-2} For a subpseudovariety \pv V,
  the \pv V-coproduct of free pro-\pv V semigroups is also free
  pro-\pv V.
\end{enumerate}

\subsection{On the injectivity of the mapping $\iota$}
\label{sec:on-inject-iota}

It is natural to ask under what conditions the natural mapping
\begin{equation*}
  \iota:\freeprod_{i\in I}S_i\to\coprod^\pv V_{i\in I}S_i
\end{equation*}
is injective for pro-\pv V semigroups $S_i$ ($i\in I$). A partial
solution to the analogous question for pseudovarieties of groups in
the category of groups can be found in
\cite[Proposition~9.1.8]{Ribes&Zalesskii:2010}. Note, that in, that
category, the free product of trivial groups is trivial, whereas the
free product of more than one trivial semigroup (in every category of
semigroups containing semigroups with more than one idempotent) is an
infinite idempotent-generated semigroup. Our main result of this
section gives a sufficient condition for injectivity of the
function~$\iota$. Our sufficient
condition holds for a large class of pseudovarieties considered in the
remainder the paper, namely equidivisible pseudovarieties containing
\pv{Sl}. We leave as an open problem the complete characterization of
the pseudovarieties for which the function $\iota$ is injective.

We start by some simple observations.

\begin{Remark}
  \label{r:mapping-factors}
  Let $\varphi_i:S_i\to T_i$ ($i\in I$) be continuous homomorphisms
  between pro-\pv V semigroups. By the universal property of \pv
  V-coproducts, there is a unique continuous homomorphism
  $\varphi_{\pv V}$ such that the following diagram commutes:
  \begin{displaymath}
    \xymatrix{
      S_i
      \ar[r]^{\varphi_i}
      \ar@{^{((}->}[d]
      &
      T_i
      \ar@{^{((}->}[d]
      \\
      \coprod^\pv V_{i\in I}S_i
      \ar[r]^{\varphi_{\pv V}}
      &
      \coprod^\pv V_{i\in I}T_i.
    }
  \end{displaymath}
\end{Remark}

\begin{Remark}
  \label{r:LG}
  Let $X$ be a set and suppose that $x$ and $y$ are distinct elements
  of~$X$ such that $x^\omega$ is a factor $y^\omega$ in the semigroup
  \Om XV. Then, \pv V is contained in~\pv{LG}. Indeed, it follows
  that, in every semigroup from \pv V, all idempotents are factors of
  each other, a property that characterizes membership in~\pv{LG}.
\end{Remark}

\begin{Lemma}
  \label{l:injectivity-for-equidivisible-on-trivial}
  Let \pv V be a pseudovariety of semigroups containing \pv J. Then
  the natural mapping $\iota:\freeprod_{i\in
    I}S_i\to\coprod^\pv V_{i\in I}S_i$ is injective whenever the $S_i$
  are trivial semigroups.
\end{Lemma}

\begin{proof}
  Let $S_i=\{e_i\}$ ($i\in I$) and let $I\to X$ be a bijection given by
  $i\mapsto x_i$. By the universal property of \pv V-coproducts, there is a
  unique continuous homomorphism $\varphi:\coprod_{i\in I}^{\pv
    V}\{e_i\}\to\Om XJ$ such that $\varphi(e_i)=x_i^\omega$.

  Given elements $e_{i_1}e_{i_2}\cdots e_{i_m}$ and
  $e_{j_1}e_{j_2}\cdots e_{j_n}$ in $\freeprod_{i\in I}\{e_i\}$
  (with adjacent indices distinct), suppose that their images under
  $\iota$ coincide. Then, their images under
  $\varphi\circ\iota$ also coincide, giving the equality
  \begin{equation}\label{eq:identity-between-idempotents}
    x_{i_1}^\omega x_{i_2}^\omega\cdots x_{i_m}^\omega
    =
    x_{j_1}^\omega x_{j_2}^\omega\cdots x_{j_n}^\omega.
  \end{equation}
  Hence, it suffices to show that, for every equality in $\Om XJ$ of
  the form~\eqref{eq:identity-between-idempotents}, with $x_k\in X$
  for every index $k$, if all adjacent indices are distinct, then we
  must have $n=m$ and $i_k=j_k$ for every $k\in\{1,\ldots,n\}$. This
  is a special case of \cite[Theorem~8.2.8]{Almeida:1994a}.
\end{proof}

\begin{Lemma}
  \label{l:S*e}
  Let \pv V be a pseudovariety of semigroups containing \pv{Sl} such
  that $\pv N\malcev\pv V=\pv V$ and let $S$ be a member of~\pv V.
  Then, the natural mapping %
  $\iota:S\freeprod\{e\}\to S\amalg^{\pv V}\{e\}$
  is injective.
\end{Lemma}

\begin{proof}
  Let $u=s_1es_2e\cdots es_m$ and $v=t_1et_2e\cdots et_n$ be distinct
  elements of~$S\freeprod\{e\}$, where
  $s_2,\ldots,s_{m-1},t_2,\ldots,t_{n-1}\in S$ and $s_1,s_m,t_1,t_n\in
  S^I$. Consider the ideal %
  $J=\bigcup_{k>\max\{m,n\}}S^Ie(Se)^kS^I$ of $S\freeprod\{e\}$. Since
  the elements $u$ and $v$ are not in~$J$, they are distinguished in
  the Rees quotient $T=S\freeprod\{e\}/J$. Note that, since $S$ is
  finite, so is $T$. To complete the proof, it suffices to establish
  that $T$ belongs to~\pv V, since then the natural homomorphism
  $S\freeprod\{e\}\to T$ factors through~$\iota$.

  Let $K$ be the complement of $S\cup\{e\}$ in $T$. Then $K$ is an
  ideal and a nilpotent subsemigroup of~$T$. Thus, if we show that
  $T/K\in\pv V$, then it follows that $T\in\pv N\malcev\pv V=\pv V$.
  But, $T/K$ is the zero sum of the semigroups $S$ and $\{e\}$, which
  is a quotient of the direct product $S\times U$ where $U$ is the
  three-element semilattice which is the zero sum of two trivial
  semigroups. This shows that indeed $T\in\pv V$.
\end{proof}

\begin{Thm}
  \label{t:injectivity-for-equidivisible}
  Let \pv V be a pseudovariety of semigroups containing \pv{Sl} such
  that $\pv N\malcev\pv V=\pv V$. Then the natural mapping
  $\iota:\freeprod_{i\in I}S_i\to\coprod^\pv V_{i\in I}S_i$ is
  injective whenever the $S_i$ are pro-\pv V semigroups.
\end{Thm}

\begin{proof}
  Consider constant homomorphisms $\varphi_i:S_i\to\{e_i\}$, the
  induced continuous homomorphism %
  $\varphi_{\pv V}:\coprod^{\pv V}_{i\in I}S_i %
  \to\coprod^{\pv V}_{i\in I}\{e_i\}$ given by
  Remark~\ref{r:mapping-factors}, and an analogous abstract
  homomorphism $\varphi:\freeprod_{i\in
    I}S_i\to\freeprod_{i\in I}\{e_i\}$. We get the following
  commutative diagram, where $\iota'$ is the natural mapping:
  \begin{displaymath}
    \xymatrix@C=14mm{
      \freeprod_{i\in I}S_i
      \ar[r]^\varphi
      \ar[d]^\iota
      &
      \freeprod_{i\in I}\{e_i\}
      \ar[d]^{\iota'}
      \\
      \coprod^{\pv V}_{i\in I}S_i
      \ar[r]^{\varphi_{\pv V}}
      &
      \coprod^{\pv V}_{i\in I}\{e_i\}.
    }
  \end{displaymath}
  Given elements $u=s_{i_1}s_{i_2}\cdots s_{i_m}$ and
  $v=t_{j_1}t_{j_2}\cdots t_{j_n}$ of $\freeprod_{i\in I}S_i$,
  with $s_i,t_i\in S_i$ and no two adjacent indices equal, suppose
  that their images under $\iota$ coincide. As $\pv J=\pv
  N\malcev\pv{Sl}\subseteq \pv N\malcev\pv V=\pv V$, we see that
  $\iota'$ is injective by
  Lemma~\ref{l:injectivity-for-equidivisible-on-trivial}. Hence, the
  images of $u$ and $v$ under $\varphi$ must be equal, too. We conclude
  that $m=n$ and $i_1=j_1,\ldots, i_m=j_m$.

  It remains to show that $s_{i_k}=t_{i_k}$ for $k=1,\ldots,m$.
  Suppose, on the contrary, that $s_j\ne t_j$ for some $j=i_k$ with
  $k\in\{1,\ldots,m\}$. Since $S_j$ is a pro-\pv V semigroup, there is
  a continuous homomorphism $\psi:S_j\to F$ into some $F\in\pv V$ such
  that $\psi(s_j)\ne\psi(t_j)$. We also consider the unique
  homomorphisms $S_i\to\{e\}$ for each $i\in I\setminus\{j\}$.
  Together, these mappings induce the horizontal mappings in the
  following commutative diagram:
  \begin{displaymath}
    \xymatrix@C=14mm{
      \freeprod_{i\in I}S_i
      \ar[r]^{\bar\psi}
      \ar[d]^\iota
      &
      F\freeprod\{e\}
      \ar[d]^{\iota'}
      \\
      \coprod^{\pv V}_{i\in I}S_i
      \ar[r]^{\bar\psi_{\pv V}}
      &
      F\amalg^{\pv V}\{e\}
    }
  \end{displaymath}
  where $\iota'$ is the natural mapping. As $\iota(u)=\iota(v)$, we
  get $\iota'(\bar\psi(u)) =\iota'(\bar\psi(v))$.
  Since $\iota'$ is injective by Lemma~\ref{l:S*e}, we conclude
  that $\bar\psi(u)=\bar\psi(v)$, which implies that
  $\psi(s_j)=\psi(t_j)$, contradicting the choice of~$\psi$.
\end{proof}

\section{The two-sided Karnofsky--Rhodes expansion}

The two-sided Karnofsky--Rhodes expansion plays a central role
in~\cite{Almeida&ACosta:2017} when studying equidivisible relatively
free profinite semigroups. We recall here this expansion.

We adopt the following notational conventions. Let $S$ be a semigroup.
Given a mapping $\varphi:A\to S$, we may denote the unique
homomorphism \mbox{$A^+\to S$} extending $\varphi$ also by $\varphi$.
Similarly, if $S$ is a pro-$\pv V$ semigroup, then the unique
continuous homomorphism $\Om AV\to S$ extending $\varphi: A\to S$ may
also be denoted by $\varphi$. Conversely, given a homomorphism
$\varphi:A^+\to S$, or a continuous homomorphism $\varphi:\Om AV\to
S$, its restriction to $A$ may also be denoted by $\varphi$, if no
confusion arises.

Let $\varphi$ be a homomorphism from $A^+$ onto a semigroup $S$. We
denote by $\Gamma_\varphi$ the \emph{two-sided Cayley graph}, whose
set of vertices is $S^I\times S^I$, and where an edge from $(s_1,t_1)$
to $(s_2,t_2)$ is a triple $((s_1,t_1),a,(s_2,t_2))$, with $a\in A$,
such that $s_1\varphi(a)=s_2$ and $t_1=\varphi(a)t_2$. We see
$\Gamma_\varphi$ as a labeled directed graph, by labeling each edge
$((s_1,t_1),a,(s_2,t_2))$ with the letter $a$. By the \emph{label} of
a directed path in $\Gamma_\varphi$ we mean the word obtained by
concatenating the successive labels of its edges.

A \emph{transition edge} of a directed graph is an edge $x\to y$ such
that there is no directed path from $y$ to~$x$.
For each path $p$ in the two-sided Cayley graph $\Gamma_\varphi$, we
denote by $T(p)$ the set of transition edges of $\Gamma_\varphi$ that
occur in~$p$. For each $u\in A^+$, let $p_u$ be the unique path
of~$\Gamma_\varphi$ from $(I,\varphi(u))$ to $(\varphi(u),I)$ that is
labeled by the word $u$. Consider the binary relation $\equiv_\varphi$
on $A^+$ defined by $u\equiv_\varphi v$ if $\varphi(u)=\varphi(v)$ and
$T(p_u)=T(p_v)$. Then one can easily check that $\equiv_\varphi$ is a
congruence, which is of finite index if both $S$ and $A$ are finite.
Consider the quotient semigroup $S_\varphi^\krho=A^+/{\equiv_\varphi}$
and the corresponding quotient homomorphism $\varphi^\krho:A^+\to
S_\varphi^\krho$.

We also consider the unique homomorphism $\pi_\varphi: S_\varphi^\krho\to S$
such that the following diagram commutes:

\begin{equation}%\label{eq:definition-of-the-expansion}
  % \begin{split}
  \xymatrix@C=12mm@R=10mm{
    &**[r]{A^+} \ar[d]^\varphi \ar[ld]_{\varphi^\krho}\\
    **[r]{S_\varphi^\krho} \ar[r]_(0.6){\pi_\varphi}
    & **[r]{S.} }
  % \end{split}
\end{equation}

\begin{Remark}\label{r:properties-of-the-projection}
  Suppose that the homomorphism $\varphi:A^+\to S$ is such that $A$
  and $S$ are finite, so that so is $S_\varphi^\krho$. It is folklore
  that $\pi_\varphi$ is an $\pv{LI}$-morphism; in fact, for $\pv W=\op
  xyz=xz\cl$, it follows easily from the definition of
  $S_\varphi^\krho$ that $\pi_\varphi$ is a $\pv{W}$-morphism, and one
  clearly has $\pv W\subseteq \pv {LI}$.
\end{Remark}

The homomorphism $\varphi^\krho$ is the \emph{two-sided
  Karnofsky--Rhodes expansion} of $\varphi$, and $S_\varphi^\krho$ is a
\emph{two-sided Karnofsky--Rhodes expansion} of $S$. The correspondence
$(S,\varphi)\mapsto (S^\krho,\varphi^\krho)$ is an expansion cut to
generators~\cite{Rhodes&Steinberg:2001}. In fact, a more general
property holds, which we state in the next proposition.

\begin{Prop}\label{p:mario-branco-proposition}
  Let $\varphi:A^+\to S$
  and $\psi: B^+\to T$
  be two homomorphisms
  from finitely generated free semigroups onto
  finite semigroups.
  Suppose that the homomorphisms $\lambda:S\to T$
  and $\alpha:A^+\to B^+$
  are such that $\lambda\circ\varphi=\psi\circ\alpha$.
  Then, there is a unique homomorphism
  $\Lambda:S_\varphi^\krho\to T_\psi^\krho$ such that the diagram
  \begin{equation}\label{eq:mario-branco-proposition-1}
  \begin{split}  
    \xymatrix@R=6mm{
      **[r]{S_\varphi^\krho}\ar[dd]_{\pi_\varphi}\ar@{-->}[rrr]^{\Lambda}
      &
      &
      &**[r]{T_\psi^\krho}\ar[dd]^{\pi_\psi}
      \\
      &A^+\ar[dl]_\varphi\ar[r]^\alpha\ar[ul]_{\varphi^\krho}
      &B^+\ar[dr]^\psi\ar[ur]^{\psi^\krho}
      &
      \\
      S\ar[rrr]^\lambda
      &
      &
      &T
    }
  \end{split}    
  \end{equation}
   is commutative.
 \end{Prop}

 The analog of Proposition~\ref{p:mario-branco-proposition} in the
 category of monoids appears as part of~\cite[Proposition
 4.4]{Branco:2006} (see also~\cite[Proposition 4.10]{Branco:2006}).
 Since the original proof is somewhat indirect, for the sake of
 completeness we present here a direct proof for the category of
 semigroups.

 \begin{proof}[Proof of Proposition~\ref{p:mario-branco-proposition}]
   Let $u,v\in A^+$.
   Suppose that $\varphi^\krho(u)=\varphi^\krho(v)$.
   We want to show that
   $\psi^\krho(\alpha(u))=\psi^\krho(\alpha(v))$.
   By the commutativity of the left triangle and of the lower
   trapezoid in Diagram~\eqref{eq:mario-branco-proposition-1},
   we know that $\psi(\alpha(u))=\psi(\alpha(v))$.
   We need to show that
   the coterminal paths $p_{\alpha(u)}$
   and $p_{\alpha(v)}$
   of $\Gamma_\psi$
   contain the same transition edges
   of $\Gamma_\psi$.

   Suppose that
   $\tau$ is a transition edge of
   $\Gamma_\psi$
   that occurs in the path $p_{\alpha(u)}$.
   Then, there are words $w_1,w_2\in B^*$
   and $b\in B$ such that
   $\alpha(u)=w_1bw_2$ and $\tau$
   is the edge
   \begin{equation*}
   \tau:(\psi(w_1),\psi(bw_2))\xrightarrow{b}(\psi(w_1b),\psi(w_2)).  
   \end{equation*}
   Moreover, there are $u_1,u_2\in A^*$, $a\in A$,
   and $w_1',w_2'\in B^*$
   such that
   $u=u_1au_2$, $w_1=\alpha(u_1)w_1'$, $\alpha(a)=w_1'bw_2'$
   and $w_2=w_2'\alpha(u_2)$. The reader may wish to refer to
   Figure~\ref{fig:comparing}.
   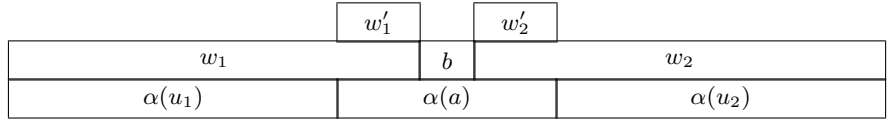
\begin{figure}[h]
     \centering
 \unitlength=2.05pt
 {\footnotesize
\begin{picture}(200,21)(-10,0)
\put(0,0){\framebox(60,7){$\alpha(u_1)$}}
%\put(25,0){\framebox(80,10){$\psi(z_2)$}}
\put(60,0){\framebox(40,7){$\alpha(a)$}}
\put(100,0){\framebox(60,7){$\alpha(u_2)$}}
\put(0,7){\framebox(75,7){$w_1$}}
\put(75,7){\framebox(10,7){$b$}}
\put(85,7){\framebox(75,7){$w_2$}}
\put(60,14){\framebox(15,7){$w_1'$}}
\put(85,14){\framebox(15,7){$w_2'$}}
%

%\put(10,10){\framebox(15,10){$\psi(y_1)$}}
\end{picture}
}
     \caption{Factorizations of $\alpha(u)$.}
     \label{fig:comparing}
   \end{figure}
   
   Note that the edge $\tau'$
   of $\Gamma_\varphi$ given by
   \begin{equation*}
   \tau':(\varphi(u_1),\varphi(au_2))\xrightarrow{a}(\varphi(u_1a),\varphi(u_2))
   \end{equation*}
   belongs to the path $p_u$ of $\Gamma_\varphi$.
   We claim that $\tau'$ is a transition edge of $\Gamma_\varphi$.
   Suppose it is not. Then, there is some $z\in A^+$
   such that
   \begin{equation*}
     \varphi(u_1az)=\varphi(u_1)\quad\text{and}\quad
     \varphi(u_2)=\varphi(zau_2).
   \end{equation*}
   By the commutativity of the lower trapezoid
   of Diagram~\eqref{eq:mario-branco-proposition-1},
   we get
   \begin{equation*}
      \psi(\alpha(u_1az))=\psi(\alpha(u_1))\quad\text{and}\quad
      \psi(\alpha(u_2))=\psi(\alpha(zau_2)).
    \end{equation*}
    Hence, we have
    \begin{align*}
      \psi(w_1b\cdot w_2'\alpha(z)w_1')&=
      \psi(\alpha(u_1)w_1'bw_2'\alpha(z)w_1')\\
      &=\psi(\alpha(u_1az)w_1')\\
      &=\psi(\alpha(u_1)w_1')\\
      &=\psi(w_1),
    \end{align*}
    and similarly
    \begin{equation*}
      \psi(w_2)=\psi(w_2'\alpha(z)w_1'\cdot bw_2).
    \end{equation*}
    Therefore, the graph $\Gamma_{\psi}$
    contains the path
   \begin{equation*}
     (\psi(w_1b),\psi(w_2))\xrightarrow{w_2'\alpha(z)w_1'}(\psi(w_1),\psi(bw_2)).
   \end{equation*}
   with contradicts $\tau$ being a transition
   edge of $\Gamma_{\psi}$.
   To avoid the contradiction,
   the edge $\tau'$ must be a transition edge
   of $\Gamma_\varphi$.
   Since the paths $p_u$
   and $p_v$
   of $\Gamma_\varphi$
   contain the same transition edges
   of $\Gamma_\varphi$,
   and in view of the commutativity
   of the lower
   trapezoid in Diagram~\eqref{eq:mario-branco-proposition-1},
   we conclude that
   \begin{equation*}
     (\psi(\alpha(u_1)),\psi(\alpha(au_2)))\xrightarrow{\alpha(a)}
     (\psi(\alpha(u_1a)),\psi(\alpha(u_2)))
   \end{equation*}
   is a path of $\Gamma_\psi$
   contained in the path
   $p_{\alpha(v)}$.
   But this path factorizes as
   {\footnotesize
   \begin{equation*}
     (\psi(\alpha(u_1)),\psi(\alpha(au_2)))
     \xrightarrow{w_1'}
     (\psi(w_1),\psi(bw_2))
     \xrightarrow{b}
     (\psi(w_1b),\psi(w_2))
     \xrightarrow{w_2'}
     (\psi(\alpha(u_1a)),\psi(\alpha(u_2))),
   \end{equation*}
 }%
 thus having $\tau$ as one of its edges.
   This shows that
   every transition edge
   of $\Gamma_\psi$
   that belongs
   to $p_{\alpha(u)}$
   also belongs to
   $p_{\alpha(v)}$.
   By symmetry,
   we conclude that
   $\psi^\krho(\alpha(u))=\psi^\krho(\alpha(v))$.
 
  Since, for every $u,v\in A^+$, the equality
  $\varphi^\krho(u)=\varphi^\krho(v)$
  implies the equality $\psi^\krho(\alpha(u))=\psi^\krho(\alpha(v))$,
  and since $\varphi^\krho$ is onto,
  we conclude that there exists a unique
  homomorphism $\Lambda$ such that
  Diagram~\eqref{eq:mario-branco-proposition-1} commutes.
\end{proof}

 A pseudovariety of semigroups $\pv V$ is said to be closed under
 two-sided Karnofsky--Rhodes expansion when $S\in\pv V$ implies
 $S_\varphi^\krho\in\pv V$, for every onto homomorphism $A^+\to S$ and
 every finite alphabet $A$. A proof of the following characterization
 of such pseudovarieties may be found in~\cite{Almeida&ACosta:2017},
 where one sees it as a direct consequence of a deep result of Rhodes
 et al.~\cite{Rhodes&Tilson:1989,Rhodes&Weil:1989b}

\begin{Thm}\label{t:characterization-of-pseudovarieties-closed-under-KR-expansion}
  A pseudovariety of semigroups $\pv V$
  is closed under two-sided Karnofsky--Rhodes expansion
  if and only if $\pv V=\pv {LI}\malcev\pv V$.
\end{Thm}

\section{KR-covers}
\label{sec:kr-cover}

We are now ready to introduce the following new definition, playing a
central role in this paper.

\begin{Def}[KR-cover of a finite semigroup]
  \label{def:KR-cover-of-finite-semigroup}
  Let $S$ be a profinite semigroup, and let $T$ be a finite semigroup.
  We say that $S$ is a \emph{KR-cover of $T$} when $T$ is a continuous
  homomorphic image of $S$ and for every continuous onto homomorphism
  $\varphi:S\to T$ there is a generating mapping $\psi:A\to T$, for
  some finite alphabet $A$ depending on $\varphi$, and a continuous
  homomorphism $\varphi_\psi:S\to T_\psi^\krho$ such that the
  following diagram commutes:
   \begin{equation}\label{eq:definition-KR-cover}
    \begin{split}
      \xymatrix@C=12mm@R=10mm{
         &S \ar[d]^\varphi \ar@{-->}[ld]_{\varphi_\psi}\\
        **[r]{T_\psi^\krho} \ar[r]_(0.6){\pi_\psi}& **[r]{T .} }
    \end{split}
  \end{equation}
\end{Def}

The generating mapping $\psi$ that appears in
Definition~\ref{def:KR-cover-of-finite-semigroup} may in fact be any
generating mapping of~$T$, as shown next. Thus, $S$ is in a sense more
general than all two-sided Karnofsky--Rhodes expansions of~$T$.

\begin{Lemma}\label{l:KR-cover-the-alphabet-may-change}
  Suppose that the profinite semigroup $S$ is a KR-cover
  of the finite semigroup $T$. Let
  $\varphi:S\to T$ be a continuous homomorphism
  from $S$ onto $T$. For every
  finite alphabet $A$ and generating
  mapping $\psi:A\to T$,
  there is a continuous homomorphism 
  $\varphi_\psi:S\to T_\psi^\krho$
  such that Diagram~\eqref{eq:definition-KR-cover} commutes.
\end{Lemma}

\begin{proof}
  Let $\psi:A^+\to T$ be any
  homomorphism from $A^+$ onto $T$, defined on a finite alphabet $A$.
  Because $S$ is a KR-cover of $T$, there is an onto
  homomorphism $\zeta:B^+\to T$, for some
  finite alphabet $B$,
  inducing a homomorphism
  $\varphi_\zeta:S\to T_\zeta^\krho$
  such that the leftmost triangle of
  the following diagram is commutative:
  \begin{equation}\label{eq:KR-cover-the-alphabet-may-change}
  \begin{split}
      \xymatrix@R=6mm{
      &
      **[r]{T_\zeta^\krho}\ar[dd]_{\pi_\zeta}\ar@{..>}[rrr]^{\Lambda}
      &
      &
      &**[r]{T_\psi^\krho}\ar[dd]^{\pi_\psi}
      \\
      S\ar[dr]_\varphi\ar[ur]^{\varphi_\zeta}
      &
      &B^+\ar[dl]_\zeta\ar@{-->}[r]^\alpha\ar[ul]_{\zeta^\krho}
      &A^+\ar[dr]^\psi\ar[ur]^{\psi^\krho}
      &
      \\
      &
      T\ar@{=}[rrr]^{\mathrm{Id}_T}
      &
      &
      &**[r]{T.}
    }
  \end{split}
  \end{equation}
  Since $\psi$ is onto, there is
  a homomorphism $\alpha:B^+\to A^+$
  such that $\zeta=\psi\circ\alpha$,
  that is, such the lower trapezoid in
  Diagram~\eqref{eq:KR-cover-the-alphabet-may-change} commutes.
  By Proposition~\ref{p:mario-branco-proposition},
  there is a homomorphism
  $\Lambda:T_\zeta^\krho\to T_\psi^\krho$
  such that Diagram~\eqref{eq:KR-cover-the-alphabet-may-change} commutes.
  Therefore, if $\varphi_\psi$ is the
  homomorphism $\Lambda\circ\varphi_\zeta$,
  then Diagram~\eqref{eq:definition-KR-cover} commutes.
\end{proof}

Letting $T$ vary in Definition~\ref{def:KR-cover-of-finite-semigroup},
we are led to the following stronger property.

\begin{Def}[KR-cover]
  \label{def:KR-cover-semigroup}
  A profinite semigroup $S$ is a \emph{KR-cover} if it is
  a KR-cover of each of its finite continuous homomorphic images.
\end{Def}

The notion of KR-cover is reminiscent of that of profinite projective semigroup,
which we recall here. Consider a pseudovariety $\pv V$ of semigroups.
A pro-\pv V semigroup $S$ is said to be \emph{\pv V-projective} if,
whenever $T$ and $R$ are pro-\pv V semigroups and $f:S\to T$ and
$g:R\to T$ are continuous homomorphisms with $g$ onto, there is
some continuous homomorphism $f':S\to R$ such that
the following diagram commutes:
\begin{equation*}
  \begin{split}
    \xymatrix{
    &S \ar[d]^f\ar@{-->}[ld]_{f'}\\
    R\ar[r]_g& **[r]{T.}}
  \end{split}
\end{equation*}
A profinite projective semigroup is just an $\pv S$-projective
semigroup, where $\pv S$ is the pseudovariety of all finite semigroups
(cf.~\cite[Remark 4.1.34]{Rhodes&Steinberg:2009qt}). Every free
pro-$\pv V$ semigroup is an example of a $\pv V$-projective semigroup.
 
The next simple observation gives our first examples of KR-covers and
provides a more precise connection with projectivity.

\begin{Prop}
  \label{p:projective-are-KR-covers}
  If\/ $\pv V$ is a semigroup pseudovariety
  closed under two-sided Karnofsky--Rhodes expansion, then every $\pv V$-projective
  semigroup is a KR-cover.
\end{Prop}

\begin{proof}
  Let $S$ be a $\pv V$-projective semigroup, and let $\varphi$ be a
  continuous homomorphism from $S$ onto a finite semigroup $T$. Note
  that $T\in\pv V$ (see, for instance,
  \cite[Proposition~3.7]{Almeida:2003cshort}).
  Consider any two-sided Karnofsky--Rhodes expansion $T_\psi^\krho$ of
  $T$, with $\psi:A\to T$ a generating mapping with finite domain. Since
  $T_\psi^\krho\in\pv V$ and $S$ is $\pv V$-projective, there is a
  continuous homomorphism $\varphi':S\to T_\psi^\krho$ such that
  $\pi_\psi\circ\varphi'=\varphi$.
\end{proof}

Proposition~\ref{p:projective-are-KR-covers}
is complemented by the next proposition.

\begin{Prop}\label{p:completely-simple-semigroups-are-KR-covers}
 Every profinite completely simple semigroup is a KR-cover.
\end{Prop}

Proposition~\ref{p:completely-simple-semigroups-are-KR-covers} follows
easily (explicit details are given below) from a simple property,
expressed in the next lemma.
Let $\pv A$ be the pseudovariety of all finite aperiodic semigroups.

\begin{Lemma}\label{l:retraction-completely-simple}
  If $\pi:S\to T$ is an onto \pv A-morphism of finite semigroups,
  and $K$ is a $\J$-class of $T$ which is a subsemigroup
  of $T$, then there is a subsemigroup $K'$
  of $S$ such that the restriction $\pi:K'\to K$
  is an isomorphism.
\end{Lemma}

\begin{proof}
  As the $\J$-class $K$ is regular,
  there is a regular $\J$-class
  $J$ of $S$ such that
  $\pi(J)=K$ and such that
  every element of $\pi^{-1}(K)$
  is a factor of the elements of $J$ (cf.~\cite[Lemma 4.6.10]{Rhodes&Steinberg:2009qt}). Moreover, since $K$ is a union of groups, $J$
  must be a union of groups, and each maximal subgroup
  of $S$ contained in $J$ is mapped by $\pi$
  onto a maximal subgroup of $T$ contained in~$K$.
  
  Fix an idempotent $e$ of~$K$.
  Let $X$ be
  the set of idempotents $\R$-equivalent to $e$,
  and let $Y$ be the set of idempotents $\L$-equivalent to $e$.
  Choose an idempotent $\gamma_e$ in $J$ such that $\pi(\gamma_e)=e$.
  If $f\in X\setminus\{e\}$ then,
  since $f=ef$, we may choose an idempotent $\gamma_f\in J$
  such that $\pi(\gamma_f)=f$ and $\gamma_f=\gamma_e\gamma_f$.
  Similarly, if $f\in Y\setminus\{e\}$
  then we may choose an idempotent $\gamma_f\in J$
  such that $\pi(\gamma_f)=f$ and $\gamma_f=\gamma_f\gamma_e$.
  Note that $X'=\{\gamma_f\mid f\in X\}$
  is contained in the $\R$-class of $\gamma_e$,
  and that $Y'=\{\gamma_f\mid f\in Y\}$
  is contained in the $\L$-class of $\gamma_e$.

  For each idempotent $g\in X'\cup Y'$, consider the $\H$-class $H_g$
  of $g$. Let $K'$ be the subsemigroup of $T$
  generated by $\bigcup_{g\in X'\cup Y'}H_g$. Then $K'$
  is a completely semigroup contained in $J$ and such that $\pi(K')\subseteq K$.
  Moreover, each $\H$-class $H$ of $K$ contains an element of the form
  $yx$, with $y\in Y$ and $x\in X$,
  and so if $H'$ is the $\H$-class of $\gamma_y\gamma_x\in K'$,
  then we have $\pi(H')=H$. Hence, we actually have $\pi(K')=K$.

  Since $K'$ has $|X|=|X'|$ $\R$-classes
  and $|Y|=|Y'|$ $\L$-classes,
  we know that $K'$ has the same number of $\H$-classes
  as $K$. On the other hand,
  the restriction of an \pv A-morphism
  to a regular $\H$-class is injective (cf.~\cite[Lemma 4.4.4]{Rhodes&Steinberg:2009qt}). We then conclude that $\pi$ restricts
  to an isomorphism $K'\to K$.
\end{proof}

\begin{proof}[Proof of Proposition~\ref{p:completely-simple-semigroups-are-KR-covers}]
  Let $S$ be a profinite completely simple semigroup, and let
  $\varphi:S\to T$ is a continuous homomorphism onto a finite
  semigroup~$T$. Then $T$ is a completely simple semigroup, and, for
  every generating mapping $\psi:A\to T$ such that $A$ is a finite
  alphabet, the homomorphism $\pi_\psi:T_\psi^\krho\to T$ is an \pv
  A-morphism (cf.~Remark~\ref{r:properties-of-the-projection}).
  Applying Lemma~\ref{l:retraction-completely-simple}, we deduce that
  there is a subsemigroup $T'$ of $T_\psi^\krho$ such that the
  restriction \mbox{$\pi_\psi|_{T'}:T'\to T$} is an isomorphism, whose inverse we
  denote $\rho$. Then the continuous homomorphism $\varphi_\psi:S\to
  T_\psi^\krho$ such that $\varphi_\psi=\rho\circ\varphi$ satisfies
  $\pi_\psi\circ\varphi_\psi=\varphi$. This establishes that $S$ is a
  KR-cover.
\end{proof}

By Proposition~\ref{p:projective-are-KR-covers},
the class of finite projective semigroups
includes examples of finite KR-covers
that are not among those provided by Proposition~\ref{p:completely-simple-semigroups-are-KR-covers}:
see~\cite[Lemma 4.1.39 and Exercise 4.1.43]{Rhodes&Steinberg:2009qt}
for simple examples of such kind.
  For a complete characterization of the finite projective semigroups, see~\cite{Zelenyuk:2003}. All finite projective semigroups are bands whose $\mathcal J$-classes
  form a chain (for a proof of this fact, alternative to~\cite{Zelenyuk:2003}
  and implicitly using the equidivisibility of the free profinite semigroup,
  see~\cite{Steinberg:2010}).

The following gives a necessary and sufficient condition for a
finite semigroup to be a KR-cover.

\begin{Prop}
  \label{p:finite-KR-covers}
  A finite semigroup $S$ is a KR-cover if and only if there are a
  finite set $A$, a generating mapping $\psi:A\to S$, and a
  homomorphism $\theta$ such that the following diagram commutes:
  \begin{equation}
    \label{eq:finite-KR-covers}
    \begin{split}
    \xymatrix@C=8mm@R=6mm{
      &S
      \ar[d]^{\mathrm{Id}}
      \ar@{-->}[ld]_{\theta}\\
      {S_\psi^\krho}
      \ar[r]_(0.6){\pi_\psi}
      & **[r]{S .} }
    \end{split}
  \end{equation}
  Moreover, then the same property holds for every generating
  mapping $\psi:A\to S$ with $A$ finite.
\end{Prop}

\begin{proof}
  The possibility of completing the diagram for every generating
  mapping $\psi:A\to S$ with $A$ finite follows directly from the
  assumption that $S$ is a finite semigroup and a KR-cover.
  Conversely, if the diagram can be completed for a generating mapping
  $\psi:A\to S$ with $A$ finite then, for every onto homomorphism
  $\varphi:S\to T$, we may consider the following commutative diagram,
  where $\theta$ is given by hypothesis and $\varphi^\krho$ by
  Proposition~\ref{p:mario-branco-proposition}:
  \begin{displaymath}
    \xymatrix@C=10mm@R=8mm{
      &S
      \ar[d]^{\mathrm{Id}}
      \ar@{-->}[ld]_{\theta}\\
      {S_\psi^\krho}
      \ar[r]_(0.6){\pi_\psi}
      \ar[d]_{\varphi^\krho}
      & S
      \ar[d]^\varphi
      \\
      T_{\psi\circ\varphi}^\krho
      \ar[r]^(0.55){\pi_{\varphi\circ\psi}}
      &T.
    }
  \end{displaymath}
  Hence, $S$ is a KR-cover.  
\end{proof}

An immediate consequence of Proposition~\ref{p:finite-KR-covers} is
that it is decidable whether a finite semigroup is a KR-cover.

If $S$ is a profinite semigroup, then the monoid $S^I$ is also a profinite semigroup, where we endow $S^I$ with the sum topology of $S$ with $\{I\}$.

\begin{Prop}\label{p:adjoining-identity-to-KR-cover}
  Let $S$ be a profinite semigroup. If $S$ is a KR-cover,
  then~$S^I$ is a KR-cover.
\end{Prop}

\begin{proof}
  Given a profinite semigroup $S$ which is a KR-cover, let
  $\Phi:S^I\to R$ be a continuous homomorphism onto a finite semigroup
  $R$. Set $T=\Phi(S)$. Denote $\varphi$ the restriction of $\Phi$
  to~$S$, which is
  an onto continuous homomorphism $S\to T$. We may take a generating
  mapping $\psi:A\to T$ such that $A$ is finite. Since $S$ is a
  KR-cover, there is a continuous homomorphism $\rho:S\to
  T_\psi^\krho$ such that $\varphi=\pi_\psi\circ\rho$.

  Consider the alphabet $B=A\cup \{b\}$,
  for some letter $b$ not in $A$. Denote $\Psi$ the extension $B\to R$
  of $\psi$ such that
  \begin{equation*}
  \Psi(b)=\Phi(I).  
  \end{equation*}
  Then $\Psi$ generates $R$. Since $R_{\Psi}^\krho$ is finite, we may
  take a positive integer~$n$ such that $x^n=x^\omega$ for every $x\in
  R_{\Psi}^\krho$. Consider the homomorphism $\alpha:A^+\to B^+$ such
  that $\alpha(a)=b^nab^n$ for every $a\in A$. By
  Proposition~\ref{p:mario-branco-proposition}, there is a
  homomorphism $\Lambda:T_\psi^\krho\to R_{\Psi}^\krho$ such that
  Diagram~\eqref{eq:an-application-of-mario-branco-proposition}
  commutes.
  \begin{equation}
    \label{eq:an-application-of-mario-branco-proposition}
    \begin{split}  
    \xymatrix@R=6mm{
      **[r]{T_\psi^\krho}\ar[dd]_{\pi_\psi}\ar@{-->}[rrr]^{\Lambda}
      &
      &
      &**[r]{R_{\Psi}^\krho}\ar[dd]^{\pi_{\Psi}}
      \\
      &A^+
        \ar[dl]_\psi\ar[r]^\alpha
        \ar[ul]_(0.45){\psi^\krho}
      &B^+
        \ar[dr]^{\Psi}
        \ar[ur]^(0.45){\Psi^\krho}
      &
      \\
      **[r]{T\ }\ar@{^{((}->}[rrr]
      &
      &
      &R
        }
    \end{split}    
  \end{equation}
  For every $a\in A^+$, we have $\Psi^\krho\circ\alpha(a)=
  \Psi^\krho(b)^\omega \cdot \Psi^\krho(a) \cdot \Psi^\krho(b)^\omega$
  and so the image of $\Psi^\krho\circ\alpha$ is contained in the
  subsemigroup $M$ of $R_{\Psi}^\krho$ defined by
  \begin{equation*}
    M=\Psi^\krho(b)^\omega\cdot
    R_{\Psi}^\krho
    \cdot
    \Psi^\krho(b)^\omega.  
  \end{equation*}
  Since $\Lambda\circ\psi^\krho=\Psi^\krho\circ\alpha$ and
  $\psi^\krho$ is onto, it follows that the image of $\Lambda$ is
  contained in~$M$. Let $\theta=\Lambda\circ\rho$. Note that $M$ is a
  monoid, whose neutral element is $\Psi^\krho(b)^n$. Therefore, the
  homomorphism $\theta:S\to M$ extends to a monoid homomorphism
  $\widetilde\theta:S^I\to M$. The reader may wish to refer to
  Diagram~\eqref{eq:solution}.
 \begin{equation}\label{eq:solution}
   \begin{split}  
     \xymatrix@R=8mm@C=14mm{
      &S
      \ar@{^{((}->}[r]
      \ar[d]^\theta
      \ar[dl]_{\rho}
      &
      S^I
      \ar[d]^\Phi
      \ar@{-->}[dl]_{\widetilde\theta}\\
      T_\psi^\krho
      \ar[r]^\Lambda
      &R_\Psi^\krho
      \ar[r]^{\pi_\Psi}
      &R
    }
    \end{split}    
  \end{equation}

  If $s\in S$, then in view of the commutativity of
  Diagram~\eqref{eq:an-application-of-mario-branco-proposition}, we
  have
  \begin{equation*}
    \pi_{\Psi}\circ\widetilde\theta(s)=
    \pi_{\Psi}\circ\Lambda\circ\rho(s)=
    \pi_{\psi}\circ\rho(s)=
    \varphi(s)=\Phi(s).
  \end{equation*}
  On the other hand, we also have
  \begin{equation*}
    \pi_{\Psi}\circ\widetilde\theta(I)=
    \pi_{\Psi}(\Psi^\krho(b)^n)
    =\Psi(b)^n=\Phi(I)^n=\Phi(I).
  \end{equation*}
  Therefore, the equality $\pi_\Psi\circ\widetilde\theta=\Phi$ holds.
  This shows that Diagram~\eqref{eq:solution} is commutative and that
  $S^I$ is a KR-cover.
\end{proof}

\section{KR-covers are equidivisible}

We highlight the following property of KR-covers, to be shown below.

\begin{Thm}
  \label{t:equidiv-kar}
  Let $S$ be a profinite semigroup. If $S$
  is a KR-cover, then it is equidivisible.
\end{Thm}

Before the proof of Theorem~\ref{t:equidiv-kar}, we formulate an
improvement of the main theorem of~\cite{Almeida&ACosta:2017}. Let
\pv{CS} be the pseudovariety of all finite completely simple
semigroups, that is, of all finite semigroups with only one (nonempty)
ideal.

\begin{Thm}
  \label{t:equidivisible-pseudovarieties}
  The following conditions are equivalent for a pseudovariety~\pv V
  not contained in~\pv{CS}:
  \begin{enumerate}
  \item\label{item:equidivisible-pseudovarieties-1}
    for every finite alphabet $A$, \Om AV is equidivisible;
  \item\label{item:equidivisible-pseudovarieties-2} for every alphabet
    $A$, \Om AV is equidivisible;
  \item\label{item:equidivisible-pseudovarieties-3} the equality
    $\pv{LI\malcev V}=\pv V$ holds;
  \item\label{item:equidivisible-pseudovarieties-4} the pseudovariety
    \pv V is closed under two-sided Karnofsky--Rhodes expansion.
  \end{enumerate}
\end{Thm}

\begin{proof}
  The equivalence of (\ref{item:equidivisible-pseudovarieties-1}) and
  (\ref{item:equidivisible-pseudovarieties-3}) is the main result
  of~\cite{Almeida&ACosta:2017}, while the equivalence of
  (\ref{item:equidivisible-pseudovarieties-3}) and
  (\ref{item:equidivisible-pseudovarieties-4}) is given by
  Theorem~\ref{t:characterization-of-pseudovarieties-closed-under-KR-expansion}.
  To finish the proof, just note that
  (\ref{item:equidivisible-pseudovarieties-4}) implies
  (\ref{item:equidivisible-pseudovarieties-2}) by
  Theorem~\ref{t:equidiv-kar} in view of
  Proposition~\ref{p:projective-are-KR-covers}, since
  free pro-\pv V semigroups are \pv V-projective.
\end{proof}

As in~\cite{Almeida&ACosta:2017}, we say that a pseudovariety \pv V is
\emph{equidivisible} if it satisfies
Property~(\ref{item:equidivisible-pseudovarieties-1}) of
Theorem~\ref{t:equidivisible-pseudovarieties}.

The proof of Theorem~\ref{t:equidiv-kar}
relies on the following lemma.

\begin{Lemma}\label{l:almost-equdivisibility}
  Let $T$ be a finite semigroup, and let
  $\psi:A^+\to T$ be an onto homomorphism, where $A$ is a finite alphabet.
  Suppose that $u,v,x,y\in T_\psi^\krho$
  are such that $uv=xy$.
  Then, there is $t\in T^I$
  such that at least one of the following situations occurs:
  \begin{enumerate}
  \item $\pi_\psi(u)t=\pi_\psi(x)$ and $\pi_\psi(v)=t\pi_\psi(y)$;
  \item $\pi_\psi(u)=\pi_\psi(x)t$ and $t\pi_\psi(v)=\pi_\psi(y)$.
  \end{enumerate}
\end{Lemma}

\begin{proof}
  Along the proof, the reader may wish to refer to the following commutative diagram:
  \begin{equation*}
      \xymatrix@C=12mm@R=10mm{
         &**[r]{A^+} \ar[d]^\psi \ar@{-->}[ld]_{\psi^\krho}\\
        **[r]{T_\psi^\krho} \ar[r]_(0.6){\pi_\psi}& **[r]{T .} }
  \end{equation*} 
  Let $\bar u,\bar v,\bar x,\bar y\in A^+$
  be such that $\psi^\krho(\bar u)=u$,
  $\psi^\krho(\bar v)=v$,
  $\psi^\krho(\bar x)=x$
  and $\psi^\krho(\bar y)=y$.
  The equality 
  \begin{equation*}
    \psi^\krho(\bar u\bar v)=\psi^\krho(\bar x\bar y)
  \end{equation*}
  means that $\psi(\bar u\bar v)=\psi(\bar x\bar y)$ and that, in the
  graph $\Gamma_{\psi}$, the coterminal paths $p_{\bar u\bar v}$ and
  $p_{\bar x\bar y}$ have the same transition edges. Since the pair
  $(\psi(\bar u),\psi(\bar v))$ is a vertex in the path $p_{\bar u\bar
    v}$ and $(\psi(\bar x),\psi(\bar y))$ is a vertex in the path
  $p_{\bar x\bar y}$, we know that at least one of the following
  situations occurs in the graph $\Gamma_{\psi}$:
  \begin{itemize}
  \item there is a (possibly empty) path
    from vertex $(\psi(\bar u),\psi(\bar v))$
    to vertex  $(\psi(\bar x),\psi(\bar y))$;
  \item there is a (possibly empty) path
    from vertex $(\psi(\bar x),\psi(\bar y))$
    to vertex  $(\psi(\bar u),\psi(\bar v))$.
  \end{itemize}
  In the first case, we have
  $\psi(\bar u)t=\psi(\bar x)$
  and
  $\psi(\bar v)=t\psi(\bar y)$, for some $t\in T^I$.
  It then suffices to note that, since $\pi_\psi\circ\psi^\krho=\psi$,
  we get $\pi_\psi(u)t=\pi_\psi(x)$ and $\pi_\psi(v)=t\pi_\psi(y)$.
  The second case is analogous.
\end{proof}

We are now ready to establish Theorem~\ref{t:equidiv-kar}.

\begin{proof}[Proof of Theorem~\ref{t:equidiv-kar}]
  Let $\{\varphi_{j,i}:S_j\to S_i\mid i,j\in I,\,i\leq j\}$ be an
  inverse system of homomorphisms between finite semigroups such that
  $S$ is its inverse limit. For each $i\in I$, the canonical
  projection $S\to S_i$ is denoted~$\varphi_i$. Let $u,v,x,y\in S$ be
  such that $uv=xy$. Take $i\in I$. Since $S$ is a KR-cover, there is
  a generating mapping $\theta:A\to S_i$, for some finite alphabet
  $A$, and a homomorphism ${(\varphi_i)}_\theta:S\to
  {(S_i)}_\theta^\krho$ such that the following diagram commutes:
  \begin{equation*}
  \xymatrix@C=15mm@R=10mm{
    &S \ar[d]^{\varphi_i}\ar@{-->}[ld]_{{(\varphi_i)}_\theta}\\
  **[r]{{(S_i)}_\theta^\krho}\ar[r]_(0.65){\pi_\theta}& **[r]{S_i .}
}
\end{equation*}
Then, we have $(\varphi_i)_\theta(uv)=(\varphi_i)_\theta(xy)$.
By Lemma~\ref{l:almost-equdivisibility},
there is $t_i\in S_i^I$
such that at least one of the following situations occurs:
\begin{enumerate}
\item $\varphi_i(u)t_i=\varphi_i(x)$ and
  $\varphi_i(v)=t_i\varphi_i(y)$;
\item  $\varphi_i(u)=\varphi_i(x)t_i$ and
  $t_i\varphi_i(v)=\varphi_i(y)$.
\end{enumerate}
Let $I_1$ (respectively, $I_2$) be the subset of elements $i$ of $I$
for which the first (respectively, second) situation occurs.
Since $I=I_1\cup I_2$, at least one of the
sets $I_1$ or $I_2$ is cofinal.
Without loss of generality, suppose
that $I_1$ is cofinal (note that, since the conjuction of
$i\in I_1$ and $k\leq i$ implies $k\in I_1$, we then actually have $I_1=I$).
By a standard compactness argument, we
conclude that there is $t\in S^I$
such that $ut=x$ and $v=ty$.
\end{proof}

The following result shows that the converse of
Theorem~\ref{t:equidiv-kar} fails.

\begin{Prop}
  \label{p:equidiv-not-KR-cover}
  Let $G^0=G\uplus\{0\}$ be the semigroup obtained by adjoining a zero
  to a finite group~$G$. Then $G^0$ is equidivisible, while $G^0$ is a
  KR-cover if and only if $G$ is trivial.
\end{Prop}

\begin{proof}
  It is easy to check that $G^0$ satisfies the definition of
  equidivisible semigroup. If $G$ is trivial, then $G^0$ is a
  two-element semilattice, which is well known to be projective,
  whence a KR-cover (cf.\
  \cite[Lemma~4.1.39]{Rhodes&Steinberg:2009qt}).

  The remainder of the proof consists in showing that $S=G^0$ is not a
  KR-cover when $G$ is a finite nontrivial group. For that purpose we
  let $\varphi:A\cup\{b\}\to S$ be a generating mapping, where
  $\varphi(A)\subseteq G$ (so that $\varphi(b)=0$). In view of
  Remark~\ref{r:properties-of-the-projection}, we know that
  $\pi_\varphi^{-1}(0)$ is a subsemigroup of $S^\krho_\varphi$
  satisfying the identity $xyz=xz$. It follows that every element of
  $\pi_\varphi^{-1}(0)$ is of one of the forms
  \begin{equation}
    \label{eq:equidiv-not-KR-cover-1}
    ubv/{\equiv_\varphi}
    \text{ or } ubvbw/{\equiv_\varphi},
    \text{ where } u,v,w\in A^*.
  \end{equation}
  Moreover, in both cases, the occurrences of $b$ label transition
  edges in the corresponding paths from $(I,0)$ to $(0,I)$: for
  instance for the first occurrence of $b$, the corresponding edge is
  of the form $(g,0)\to(0,s)$ with $g\in G$ and, therefore, it must be
  a transition edge; the argument is similar for the last occurrence
  of $b$ taking into account instead the second component of the
  vertices. Hence, the idempotents in $\pi_\varphi^{-1}(0)$ are the
  elements of the second form in~(\ref{eq:equidiv-not-KR-cover-1}).

  Next, we show that $ubvbwt\equiv_\varphi ubvbw$ with $u,v,w\in A^*$
  and $t\in A^+$ implies $\varphi(t)=1$. Indeed, as the paths
  $p_{ubvbwt}$ and $p_{ubvbw}$ use the same transition edges and the
  $b$'s label such edges, comparing the second components of the end
  vertex of the edge corresponding to the second $b$, we conclude that
  $\varphi(wt)=\varphi(w)$. Hence, $\varphi(t)$ is equal to the
  identity element of the group $G$.
  
  Suppose that there is a homomorphism $\theta:S\to S_\varphi^\krho$
  completing Diagram~(\ref{eq:finite-KR-covers}).  Let $g\in G$. As
  $\pi_\varphi(\theta(g))=g\ne0$, there is some $t\in A^+$ such that
  $\theta(g)=t/{\equiv_\varphi}$. Moreover, we have $\theta(0)\theta(g)=\theta(0)$.
  Since $\theta(0)$ is an idempotent in $\pi_\varphi^{-1}(0)$, we already know
  that it is of the form $ubvbw/{\equiv_\varphi}$ for some $u,v,w\in A^*$.
  It follows from the previous paragraph
  that $g=\pi_\varphi(\theta(g))=\pi_\varphi(t/{\equiv_\varphi})=\varphi(t)$ is the identity of $G$. This shows that $G$ is trivial.
\end{proof}

It is routine to check that an inverse quotient limit of equidivisible compact
semigroups is equidivisible. In the context of this paper, it is
worthy to record the following similar fact.

\begin{Prop}\label{p:inverse-limit-of-KR-covers}
  An inverse quotient limit of KR-covers is a KR-cover.
\end{Prop}

\begin{proof}
  Let $S=\varprojlim_{i\in I} S_i$ be an inverse quotient limit of the
  KR-covers $S_i$. For each $i\in I$, let $p_i$ be the canonical
  projection $S\to S_i$. Suppose that $\varphi:S\to T$ is a continuous
  homomorphism onto a finite semigroup $T$. Then there is $k\in I$ for
  which there is a factorization $\varphi=\varphi_k\circ p_k$ such
  that $\varphi_k:S_k\to T$ is a continuous onto homomorphism (see,
  for instance, \cite[Lemma~3.1.37]{Rhodes&Steinberg:2009qt}). As
  $S_k$ is a KR-cover, there is a finite alphabet $A$ and a generating
  mapping $\psi:A\to T$ for which there is a continuous homomorphism
  $\rho:S_k\to T_\psi^\krho$ satisfying $\varphi_k=\pi_\psi\circ\rho$.
  Since the continuous homomorphism $\varphi_\psi=\rho\circ p_k$
  satisfies $\pi_\psi\circ\varphi_\psi=\varphi$, we conclude that $S$
  is a KR-cover. The diagram
  \begin{displaymath}
    \xymatrix@C=18mm{
      S
      \ar[r]^\varphi
      \ar[d]_{p_k}
      \ar[rd]^(.75){\varphi_\psi}|!{[d];[r]}\hole
      &
      T
      \\
      S_k
      \ar[r]^\rho
      \ar[ru]^(.25){\varphi_k}
      &
      *[r]{\ T_\psi^\krho.}
      \ar[u]_{\pi_\psi}
    }
  \end{displaymath}
  may help visualizing the various homomorphisms involved in this
  proof.
\end{proof}

To finish this section we observe that a profinite KR-cover may not
be, up to isomorphism, an inverse limit of finite KR-covers. Indeed,
KR-covers are equidivisible by Theorem~\ref{t:equidiv-kar}. Now, by
\cite[Theorem~1.9]{McKnight&Storey:1969} (which is attributed to
Rees), elements of finite order of an equidivisible semigroup lie in
groups; in particular, finite
KR-covers are unions of groups and, therefore, so are their inverse
limits. On the other hand, the semigroups \Om AV of
Theorem~\ref{t:equidivisible-pseudovarieties} (that is, with \pv V
closed under two-sided Karnofsky--Rhodes expansion) are KR-covers by
Proposition~\ref{p:projective-are-KR-covers} but they are never unions
of groups since \pv V contains~\pv{LI}.

\section{Profinite coproducts of KR-covers}
 
In combination with Theorem~\ref{t:equidiv-kar}, the following
property provides a way of producing new examples of profinite
equidivisible semigroups.

\begin{Thm}
  \label{t:coproduct-KR-covers}
  For every pseudovariety of semigroups \pv V closed under two-sided
  Karnofsky--Rhodes expansion, the class of all pro-\pv V KR-covers is
  closed under \pv V-coproducts.
\end{Thm}

\begin{proof}
  Let $(S_i)_{i\in I}$ be a family of pro-\pv V KR-covers. Let
  $S$ be their \pv V-coproduct,
  with associated continuous
  homomorphisms $\varphi_i:S_i\to S$.

  Consider a continuous homomorphism $\psi:S\to T$ onto a finite
  semigroup~$T$ and let $\psi_i=\psi\circ\varphi_i$.
  Let $T_i$ be the image of $\psi_i$.
  Since $S_i$ is a
  KR-cover,
  there are a finite set $A_i$,
  a generating mapping $\delta_i:A_i\to T_i$,
  depending on $T_i$ only (not on $i$),
  and a continuous homomorphism $\beta_i$ such that
  the following diagram commutes:
  \begin{equation}
  \begin{split}
    \label{eq:slicing-T}
    \xymatrix@C=15mm{
      & S_i \ar[ld]_{\beta_i} \ar[d]^{\psi_i} \ar[r]^{\varphi_i}
      & S \ar[d]^\psi \\
      **[r]{{(T_i)}_{\delta_i}^\krho} \ar[r]^(.7){\pi_{\delta_i}}
      & T_i \ar@{^{((}->}[r] 
      & **[r]{T .}}
  \end{split}  
  \end{equation}
  We may assume that $T_i\neq T_j$ implies $A_i\cap A_j=\emptyset$,
  and we let $A=\bigcup_{i\in I}A_i$. Since $T$ is finite, the set
  $\{T_i\mid i\in I\}$ is finite; moreover, its union is~$T$ since the
  union of the images of the $\varphi_i$ generates a dense
  subsemigroup of~$S$. The union $\delta=\bigcup_{i\in I}\delta_i$ is
  then a generating mapping $A\to T$ with finite domain.

  By Proposition~\ref{p:mario-branco-proposition}, there are homomorphisms
  $\eta_i:{(T_i)}_{\delta_i}^\krho\to T_\delta^\krho$,
  with $\eta_i$ depending only on $T_i$,
  such that the
  lower rectangle of the following diagram commutes:
  \begin{equation}
  \begin{split}
    \label{eq:KR-closure}
        \xymatrix@C=13mm{
      & S \ar[rd]^\psi \ar@{-->}[d]_\beta
      & \\
      S_i \ar[ru]^{\varphi_i} \ar[rd]_{\beta_i}
      & **[r]{T_\delta^\krho} \ar[r]^{\pi_\delta}
      & T \\
      & **[r]{{(T_i)}_{\delta_i}^\krho} \ar[r]^(.63){\pi_{\delta_i}}
        \ar[u]_{\eta_i}
      & **[r]{T_i .} \ar@{^{((}->}[u]
    }
  \end{split}  
  \end{equation}
  Since $T$ is a finite continuous homomorphic image of the pro-$\pv
  V$ semigroup~$S$, we know that $T$ belongs to $\pv V$ (again, see,
  for instance, \cite[Proposition~3.7]{Almeida:2003cshort}). And since
  $\pv V$ is closed under two-sided Karnofsky--Rhodes expansion,
  $T_\delta^\krho$ also belongs to $\pv V$. Therefore, by the
  definition of $\pv V$-coproduct, the homomorphisms
  $\eta_i\circ\beta_i$ ($i\in I$) induce a unique continuous
  homomorphism $\beta:S\to T_\delta^\krho$ such that the left triangle
  in Diagram~\eqref{eq:KR-closure} commutes for each~$i\in I$. Then,
  taking also into account the commutativity of
  Diagram~\eqref{eq:slicing-T}, we deduce that for every $i\in I$ and
  every $s\in S_i$, the following chain of equalities holds:
  \begin{equation*}
    \psi\circ\varphi_i(s)=\psi_i(s)=\pi_{\delta_i}\circ\beta_i(s)
    =\pi_\delta\circ\eta_i\circ\beta_i(s)
    =\pi_\delta\circ\beta\circ\varphi_i(s).
  \end{equation*}
  Since $\bigcup_{i\in I}\varphi_i(S_i)$ generates a dense subsemigroup
  of~$S$, we conclude that
  $\psi=\pi_\delta\circ\beta$ (that is,  Diagram~\eqref{eq:KR-closure}
  commutes). This completes the proof that $S$ is a
  KR-cover.
\end{proof}

For a set $A$, one may consider the $A$-indexed \pv V-coproduct
$\coprod_{a\in A}^{\pv V}\{1\}$ of trivial semigroups. Note that it is
precisely the free object on $A$ in the category of
idempotent-generated pro-\pv V semigroups. By
Theorem~\ref{t:coproduct-KR-covers}, such semigroups are KR-covers
whenever \pv V is closed under two-sided Karnofsky--Rhodes expansion,
whence they are equidivisible by Theorem~\ref{t:equidiv-kar}.

\section{Letter super-cancellative equidivisible profinite semigroups}

In this section we completely characterize a class of equidivisible
profinite semigroups, defined by a cancellation property
(Definition~\ref{def:letter-super-cancellative}), that was considered
in~\cite{Almeida&ACosta:2017, Almeida&ACosta&Costa&Zeitoun:2019}.

\subsection{Letter super-cancellative semigroups}

In the following definition, we adopt the terminology
of~\cite{Almeida&ACosta:2017}.

\begin{Def}[Letter super-cancellative
  semigroup]\label{def:letter-super-cancellative}
  Let $S$ be a compact semigroup and suppose that $S$ is generated, as
  a topological semigroup, by a finite subset~$A$. Say that $S$ is
  \emph{letter super-cancellative} (with respect to~$A$) when, for
  every $a,b\in A$ and $u,v\in S^I$, the following holds: if we have
  $ua=vb$ or $au=bv$, then we have $a=b$ and $u=v$.
\end{Def}

\begin{Remark}\label{r:slc-independence-from-alphabet}
  As observed in~\cite[Lemma 6.1]{Almeida&ACosta&Costa&Zeitoun:2019},
if $S$ is letter super-cancellative with respect to $A$ and
also with respect to $B$, then $A=B$.   In \cite{Almeida&ACosta&Costa&Zeitoun:2019}, a letter super-cancellative semigroup is called \emph{finitely cancellable}.
\end{Remark}

\begin{Examp}
  \label{eg:equidivisible-pseudovariety}
  By~\cite[Proposition~6.3]{Almeida&ACosta:2017}, for an equidivisible
  pseudovariety \pv V not contained in~\pv{CS} and a finite set $A$,
  \Om AV is letter super-cancellative. In view of
  Theorem~\ref{t:equidivisible-pseudovarieties}, this holds precisely
  when \pv V is closed under two-sided Karnofsky--Rhodes expansion.
\end{Examp}

An \emph{epigroup} is a semigroup $S$ such that every element $x$
of~$S$ has some power $x^n$ lying in a subgroup of $S$, with $n$ a
positive integer. For example, finite semigroups and
completely simple semigroups are epigroups.
It is easy to see that no profinite epigroup is letter super-cancellative.
The argument extends to the following proposition.

\begin{Prop}
  \label{p:completely-simple-in-coproduct-non-cancellable}
  Let \pv V be a pseudovariety containing~\pv{Sl}. If a nonempty
  family of nontrivial pro-$\pv V$ semigroups includes some epigroup
  and the $\pv V$-coproduct of the family is finitely generated as a
  topological semigroup, then that \pv V-coproduct is not letter
  super-cancellative.
\end{Prop}

\begin{proof}
  Consider a nonempty family $(S_i)_{i\in I}$
  of nontrivial semigroups.
  Let $i_0\in I$ be such that $S_{i_0}$ is an epigroup.
  Let $A$ be a finite generating subset of
  the profinite semigroup
  $S=\coprod^\pv V_{i\in I}S_i$.
  By Lemma~\ref{l:traces-of-generating-sets}
  there is $a\in A\cap S_{i_0}$.
  Since $S_{i_0}$ is an epigroup, there is a positive integer $k$
  such that $a^k=a^{\omega+k}$.
  Since $a^k\cdot I=a^k\cdot a^{\omega-k}$
  but $I\neq a^{\omega-k}$,
  we conclude that $S$ is not letter super-cancellative with respect to $A$.
\end{proof}

\subsection{Strong KR-covers}
\label{sec:gkr-cover}

The following somewhat subtly strengthened version of KR-cover is
crucial in our main result of this section.

\begin{Def}[Strong KR-cover]
  Consider a profinite semigroup with a generating mapping
  $\kappa:A\to S$ such that $A$ is finite. Let $T$ be a continuous
  finite homomorphic image of $S$. We say that $S$ is a \emph{strong
    KR-cover of $T$ with respect to $\kappa$} if, for every continuous
  onto homomorphism $\varphi:S\to T$, there is a continuous
  homomorphism $\varphi_\kappa:S\to T_{\varphi\circ\kappa}^\krho$ such
  that the following diagram commutes:
\begin{equation}\label{eq:definition-StrongKR-cover}
\begin{split}
  \xymatrix@C=14mm@R=11mm{
    A\ar[r]^\kappa
    \ar[d]_{(\varphi\circ\kappa)^\krho}
    &S\ar[d]^\varphi\ar@{-->}[ld]_{\varphi_\kappa}\\
    **[r]{T_{\varphi\circ\kappa}^\krho}\ar[r]_(0.65){\pi_{\varphi\circ\kappa}}
    &
    **[r]{T.}
  }
\end{split}
\end{equation}
The profinite semigroup $S$ is a \emph{strong KR-cover of $T$}
if it is a strong KR-cover of $T$
with respect to some such $\kappa$.
Finally, $S$ is a \emph{strong KR-cover} if
it is a strong KR-cover of each of its finite continuous homomorphic images.
\end{Def}

\begin{Remark}
  \label{rmk:StrongKR-vs-KR}
  Every strong KR-cover (of a finite semigroup $T$)
  is a KR-cover (of $T$).
\end{Remark}

The definition of strong KR-cover is motivated by the following link
with the property of being letter super-cancellative.

\begin{Prop}\label{p:StrongKR-covers-are-finitely-cancellable}
  Let $S$ be a profinite semigroup with a generating mapping
  $\kappa:A\to S$, where $A$ is a finite alphabet. If $S$ is a strong
  KR-cover of the trivial semigroup with respect to~$\kappa$, then $S$
  is letter super-cancellative with respect to $\kappa(A)$.
\end{Prop}

\begin{proof}
  Let $x,y\in S$ and $a,b\in A$
  be such that $x\cdot\kappa(a)=y\cdot \kappa(b)$.
  We want to show that $x=y$ and that $\kappa(a)=\kappa(b)$.  
  Since,  by Theorem~\ref{t:equidiv-kar}, the
  semigroup $S$ is equidivisible,
  we know that there is $t\in S^I$
  such that
  $xt=y$ and
  $\kappa(a)=t\kappa(b)$,
  or such
  that
  $x=ty$
  and 
  $\kappa(a)t=\kappa(b)$.
   Without loss generality, we assume that
  $xt=y$ and
  $\kappa(a)=t\kappa(b)$.

  Arguing by contradiction, suppose $t\neq I$.
  Let $\varphi:S\to T$ be
  the continuous homomorphism from $S$ onto
  the trivial semigroup $T$.
  As~$S$ is a strong KR-cover of $T$, there is a
 continuous
    onto homomorphism
  $\varphi_\kappa:S\to T_{\varphi\circ\kappa}^\krho$
  such that Diagram~\ref{eq:definition-StrongKR-cover} commutes.
  Because $t\neq I$, we know that
  $\varphi_\kappa(\kappa(a))$ belongs to $T_{\varphi\circ\kappa}^\krho\cdot
\varphi_\kappa(\kappa(b))$.
Therefore,
and since $\varphi_\kappa\circ\kappa=(\varphi\circ\kappa)^\krho$,
there are $c\in A$ and $u\in A^*$
such that
$(\varphi\circ\kappa)^\krho(a)=(\varphi\circ\kappa)^\krho(cub)$.
The latter equality
means that
$\varphi(\kappa(a))=\varphi(\kappa(cub))$
and that
the coterminal paths
$p_{a}$
and $p_{cub}$
of the two-sided Cayley graph $\Gamma_{\varphi\circ\kappa}$
have the same transition edges.
But $p_a$
has length one,
while $p_{cub}$ has at least two distinct transition edges
of $\Gamma_{\varphi\circ\kappa}$, namely
the edges
\begin{equation*}
(I,\varphi(\kappa(cub)))\xrightarrow{c}
(\varphi(\kappa(c)),\varphi(\kappa(ub)))
\end{equation*}
and
\begin{equation*}
(\varphi(\kappa(cu)),\varphi(\kappa(b)))\xrightarrow{b}
(\varphi(\kappa(cub)),I).
\end{equation*}
We reached a contradiction, resulting
from assuming that $t\neq I$.
This shows that indeed we have~$x=y$ and $\kappa(a)=\kappa(b)$.

Symmetrically, if
$\kappa(a)\cdot x=\kappa(b)\cdot y$ holds,
then $x=y$ and $\kappa(a)=\kappa(b)$.
\end{proof}

\begin{Remark}
  \label{r:stonr-KR-cover-dependence-on-alphabet}
  In view of
  Proposition~\ref{p:StrongKR-covers-are-finitely-cancellable} and
  Remark~\ref{r:slc-independence-from-alphabet}, up to the name of
  generators, there can be only one injective generating mapping
  $\kappa:A\to S$  with respect to which the profinite semigroup $S$
  is a strong KR-cover.
\end{Remark}

Note that no finite semigroup is a strong KR-cover: indeed, finite
semigroups are epigroups and we already observed that epigroups are
not letter super-cancellative, while strong KR-covers of the trivial
semigroup are letter super-cancellative by
Proposition~\ref{p:StrongKR-covers-are-finitely-cancellable}. On the
other hand, there are several examples of finite KR-covers (see
Section~\ref{sec:kr-cover}).

More generally, for every pseudovariety of semigroups \pv V containing
\pv{Sl} and closed under two-sided Karnofsky--Rhodes expansion, if $S$
is the $\pv V$-coproduct of a nonempty finite family of finitely
generated pro-$\pv V$ semigroups which are KR-covers, with at least
one being an epigroup, then $S$ is a KR-cover which is not a strong
KR-cover, thanks to Theorem~\ref{t:coproduct-KR-covers} and also
Propositions~\ref{p:completely-simple-in-coproduct-non-cancellable}
and~\ref{p:StrongKR-covers-are-finitely-cancellable}.

Next is a complete characterization of the strong KR-covers.

\begin{Thm}
  \label{t:characterization-StrongKR}
  Let $S$ be a finitely generated profinite semigroup.
  The following conditions are equivalent:
  \begin{enumerate}
  \item $S$ is equidivisible and letter super-cancellative;\label{item:characterization-StrongKR-1}
  \item $S$ is a strong KR-cover;\label{item:characterization-StrongKR-2}
  \item $S$ is a KR-cover, and $S$ is a strong KR-cover of the trivial semigroup.\label{item:characterization-StrongKR-3}
  \end{enumerate}
\end{Thm}

The proof of Theorem~\ref{t:characterization-StrongKR}, given in this
section, is inspired by the proof of~\cite[Theorem
8.3]{Almeida&ACosta:2017}.

Note that, for an equidivisible pseudovariety \pv V not contained
in~\pv{CS}, Property~(\ref{item:characterization-StrongKR-1}) of
Theorem~\ref{t:characterization-StrongKR} holds for \Om AV whenever
$A$ is finite (cf.\ Example~\ref{eg:equidivisible-pseudovariety}).
Thus, strong KR-covers may be seen as a generalization of finitely
generated free profinite semigroups over pseudovarieties closed under
two-sided Karnofsky--Rhodes expansion.

We recall the following definition used in~\cite{Almeida&ACosta:2017}.

\begin{Def}[Transition edge for a pseudoword]
  Let $A$ be a finite alphabet and $u\in\Om AS$.
Consider a continuous homomorphism $\varphi$ from $\Om AS$
onto a finite semigroup.
Suppose that $(u_n)_n$ is a sequence of elements of $A^+$
converging to $u$.
A \emph{transition edge} for $u$ in $\Gamma_\varphi$
is an edge of $\Gamma_\varphi$ which is a transition edge for $u_n$ in $\Gamma_\varphi$ for all sufficiently large $n$.

Moreover, a sequence of edges of $\Gamma_\varphi$ is said to be a
\emph{sequence of transition edges for $u$} in $\Gamma_\varphi$ if it
is the sequence of all transition edges of $u_n$ for all sufficiently
large $n$, where $(u_n)_n$ is a sequence of elements of $A^+$ such
that $u_n\to u$.
\end{Def}

Since $\varphi^\krho(u_n)\to\varphi^\krho(u)$,
the property of being a transition edge  (or of being a sequence of transition edges) for $u$ in $\Gamma_\varphi$ does not depend on the choice of the sequence $u_n$, it only depends on $\varphi$ and $u$.

For proving Theorem~\ref{t:characterization-StrongKR}, we need the
following property, contained in~\cite[Lemma 8.1]{Almeida&ACosta:2017}.

\begin{Lemma}
  \label{l:factorization-around-a-transition-edge}
  Let $\varphi$ be a continuous homomorphism from $\Om AS$ onto a
  finite semigroup, where $A$ is a finite alphabet. Let $u\in\Om AS$.
  If $((s_1,t_1),a,(s_2,t_2))$ is a transition edge for $u$
  in~$\Gamma_\varphi$, then there is a factorization $u=u_1au_2$
  of~$u$, with $u_1,u_2\in (\Om AS)^I$, such that $\varphi(u_1)=s_1$
  and $\varphi(u_2)=t_2$.
 \end{Lemma}

\begin{proof}[Proof of Theorem~\ref{t:characterization-StrongKR}]
  The implication \eqref{item:characterization-StrongKR-2}
  $\Rightarrow$ \eqref{item:characterization-StrongKR-3} is trivial,
  and \eqref{item:characterization-StrongKR-3} $\Rightarrow$
  \eqref{item:characterization-StrongKR-1} follows from
  Theorem~\ref{t:equidiv-kar} and
  Proposition~\ref{p:StrongKR-covers-are-finitely-cancellable}. It
  remains to
  show~\mbox{\eqref{item:characterization-StrongKR-1}$\Rightarrow$
    \eqref{item:characterization-StrongKR-2}}. Suppose that $S$ is
  letter super-cancellative with respect to the finite set~$A$. Denote
  by $\kappa$ the continuous onto homomorphism $\Om AS\to S$ extending
  the inclusion of $A$ in~$S$.
  
  Let $\varphi:S\to T$ be a continuous homomorphism onto a finite
  semigroup. Let $\Phi=\varphi\circ\kappa$. We then have the following
  diagram, where the outer square commutes.
  \begin{equation}
    \label{eq:characterization-StrongKR}
    \begin{split}
      \xymatrix@C=14mm@R=11mm{
      \Om AS\ar[r]^{\kappa}
      \ar[d]_{\Phi^\krho}
      &S\ar[d]^{\varphi}\ar@{-->}[dl]_{\varphi_\kappa}
      \\
      **[r]{T_\Phi^\krho}
      \ar[r]_(.6){\pi_\Phi}
      &T
        }
    \end{split}
  \end{equation}
  Our aim is to show that there exists a continuous homomorphism
  $\varphi_\kappa$ such that the whole diagram commutes. For that
  purpose, take $u,v\in\Om AS$ such that $\kappa(u)=\kappa(v)$. We
  claim that $\Phi^\krho(u)=\Phi^\krho(v)$. Let $(\varepsilon_i)_{i\in
    \{1,\ldots,n\}}$ and $(\delta_i)_{i\in \{1,\ldots,m\}}$ be the
  sequences of transition edges in $\Gamma_{\Phi}$ respectively for
  $u$ and for $v$. Without loss of generality, we may assume
  that~$n\leq m$. For an edge $\varepsilon$ of the graph
  $\Gamma_\Phi$, $\alpha(\varepsilon)$ and $\omega(\varepsilon)$
  denote the beginning and end vertices of~$\varepsilon$,
  respectively.

    Suppose that the set
  \begin{equation}\label{eq:characterization-0}
       \{i\in \{1,\ldots,n\}\mid \varepsilon_i\neq \delta_i\}
  \end{equation}
  is nonempty, and let
  $j$ be its minimum.
  By Lemma~\ref{l:factorization-around-a-transition-edge},
  there are factorizations
  \begin{equation*}
  u=u_1au_2\quad\text{ and }\quad v=v_1bv_2
  \end{equation*}
  with $a,b\in A$ and $u_1,u_2,v_1,v_2\in (\Om AS)^I$,
  such that
  \begin{equation*}
  \varepsilon_j=((\Phi(u_1),\Phi(au_2)),a,(\Phi(u_1a),\Phi(u_2))    
  \end{equation*}
  and
  \begin{equation*}
  \delta_j=((\Phi(v_1),\Phi(bv_2)),b,(\Phi(v_1b),\Phi(v_2)).    
\end{equation*}
Note that $\alpha(\varepsilon_j)$ and $\alpha(\delta_j)$
belong to the same strongly connected component of $\Gamma_{\Phi}$, by the minimality of the index $j$.

If $u_1=I$, then $j=1$, which in turn implies that
$v_1=I$.
We then have $\kappa(au_2)=\kappa(u)=\kappa(v)=\kappa(bv_2)$.
Because $S$ is finitely cancellable with respect to $A$,
we deduce that $a=b$ and $\kappa(u_2)=\kappa(v_2)$, and
so we get that $\varepsilon_j=\delta_j$, a contradiction.
Hence we have $u_1\neq I$, and, analogously, $v_1\neq I$.

Similarly, if $u_2=v_2=I$, then $j=n=m$ and $\varepsilon_j=\delta_j$, a contradiction.

Suppose that $u_2=I$ and $v_2\neq I$.
We then have $\kappa(u_1a)=\kappa(v_1bv_2)$.
Since~$S$ is letter super-cancellative with respect to $A$,
it follows that there is a pseudoword $v_2'\in(\Om AS)^I$
such that $v_2=v_2'a$
and $\kappa(u_1)=\kappa(v_1bv_2')$.
This implies the existence of a path 
in $\Gamma_\Phi$
from 
$\omega(\delta_j)=(\Phi(v_1b),\Phi(v_2))$
to
$\alpha(\varepsilon_j)=(\Phi(v_1bv_2'),\Phi(a))$
labeled by a word $v_2''$ of $A^*$
such that $\Phi(v_2'')=\Phi(v_2')$.
As $\alpha(\varepsilon_j)$ and
  $\alpha(\delta_j)$ belong to the same strongly connected component
  of $\Gamma_\Phi$, we deduce that there is in $\Gamma_\Phi$
  a path from $\omega(\delta_j)$
  to $\alpha(\delta_j)$,
  contradicting the fact that $\delta_j$
  is a transition edge of $\Gamma_\Phi$.
  Therefore, we must have $u_2\neq I$.

Since $S$ is equidivisible,
and we have $\kappa(u_1a\cdot u_2)=\kappa(v_1\cdot bv_2)$
with none of the elements
$\kappa(u_1a)$, $\kappa(u_2)$, $\kappa(v_1)$, $\kappa(bv_2)$ being equal to
$I$, we know that  there is $t\in (\Om AS)^I$ such that
  \begin{equation}\label{eq:characterization-1}
    \kappa(u_1at)=\kappa(v_1)\quad \text{and}\quad
    \kappa(u_2)=\kappa(tbv_2),
  \end{equation}
  or
  \begin{equation}\label{eq:characterization-2}
    \kappa(v_1t)=\kappa(u_1a)\quad \text{and}\quad \kappa(bv_2)=\kappa(tu_2).
  \end{equation}

  If Case~\eqref{eq:characterization-1} holds, then there is
  in $\Gamma_\Phi$ a (possibly empty) path from
  $\omega(\varepsilon_j)$
  to
  $\alpha(\delta_j)$, labeled by a word
  $t_0\in A^\ast$ such that $\Phi(t_0)=\Phi(t)$.
  But,
  since
  $\alpha(\varepsilon_j)$
  and
  $\alpha(\delta_j)$
  are in the same strongly connected component,
  we reach a contradiction with the hypothesis
that $\varepsilon_j$ is a transition edge of $\Gamma_\Phi$.

    Therefore, Case~\eqref{eq:characterization-2} holds
  with $t\neq I$.
  Since $S$ is letter super-cancellative with respect to~$A$, there is $t'\in(\Om AS)^I$ with $t=t'a$
  and
  \begin{equation}\label{eq:characterization-3}
    \kappa(v_1t')=\kappa(u_1)\quad \text{and}\quad \kappa(bv_2)=\kappa(t'au_2).
  \end{equation}
  
  Suppose that $t'\neq I$. Again
  because $S$ is letter super-cancellative with respect
  to~$A$, it follows from~\eqref{eq:characterization-3}
  that 
  there is $t''\in(\Om AS)^I$ with $t'=bt''$
  and
  \begin{equation}\label{eq:characterization-4}
    \kappa(v_1b\cdot t'')=\kappa(u_1)
    \quad
    \text{and}
    \quad
    \kappa (v_2)=\kappa(t''\cdot au_2).
  \end{equation}
  This implies the existence of a path in $\Gamma_{\Phi}$
  from $\omega(\delta_j)$ to $\alpha(\varepsilon_j)$,
  which once more contradicts the definition of a transition edge.

  Therefore, we have $t'=I$, and so,
  once again because
  $S$ is letter super-cancellative with respect
  to $A$, from~\eqref{eq:characterization-3}
  we get $\kappa(v_1)=\kappa(u_1)$, $a=b$ and $\kappa(v_2)=\kappa(u_2)$.
  This yields $\varepsilon_j=\delta_j$, which contradicts the initial
  assumption.
  Therefore, the set~\eqref{eq:characterization-0} is empty.
  In particular, $\varepsilon_n=\delta_n$ holds.
  Since $\varepsilon_n$ is the last transition edge for $u$,
  we have $\omega(\delta_n)=(\Phi(u),I)$, which means that
  $\delta_n$ is the last transition edge for $v$,
  whence $m=n$ and $\varepsilon_i=\delta_i$ for every $i\in\{1,\ldots,n\}$.

  We have therefore established the claim that $\Phi^\krho(u)=\Phi^\krho(v)$
  holds whenever $\kappa(u)=\kappa(v)$,
  and so there is a unique continuous
  homomorphism $\varphi_\kappa:S\to T_\Phi^\krho$
  such that Diagram~\ref{eq:characterization-StrongKR}
  commutes, thus showing that $S$ is a strong KR-cover.
\end{proof}

\subsection{A strong KR-cover which is not relatively free}

Let $S$ be a finite semigroup and let $\varphi:A\to S$ be a generating mapping,
where $A$ is a finite alphabet. We define the onto homomorphism
$\varphi^{\krho^n}:A^+\to S_\varphi^{\krho^n}$ recursively by
\begin{equation*}
  \varphi^{\krho^0}=\varphi
  \quad
  \text{ and }
  \quad
  \varphi^{\krho^{n+1}}=(\varphi^{\krho^n})^{\krho}\quad (n\geq 0).
\end{equation*}
For $m\geq n$, let $\varrho_{m,n}$ be the unique (onto)
homomorphism $S_\varphi^{\krho^m}\to S_\varphi^{\krho^n}$
such that the  following diagram commutes:
\begin{equation*}
  \xymatrix@C=2mm{
    & A^+ \ar[dr]^{\varphi^{\krho^m}} \ar[dl]_{\varphi^{\krho^n}}&\\
    {S_\varphi^{\krho^n}}
    && {S_\varphi^{\krho^m} .} \ar[ll]^{\varrho_{m,n}}
  }
\end{equation*}
The family of homomorphisms $\{\varrho_{m,n}\mid m,n\in\mathbb
N,m\geq n\}$ defines an inverse system of $A$-generated semigroups.
Consider its inverse limit, the profinite semigroup
$S_\varphi^{\krho^\omega}=\varprojlim S_\varphi^{\krho^n}$, with
generating mapping $\varphi^{\krho^\omega}:A\to
S_\varphi^{\krho^\omega}$. The associated projection
$S_\varphi^{\krho^\omega}\to S_\varphi^{\krho^n}$, denoted
$\varrho_{n}$, is an onto continuous homomorphism, since the
connecting homomorphisms defining the inverse limit are onto (see,
for instance, \cite[Lemma~3.1.26]{Rhodes&Steinberg:2009qt}).

Using results that can be found in~\cite{Rhodes&Steinberg:2009qt}
(namely, Corollary~5.3.22 and Theorem~3.6.4), one may show that for an
arbitrary pseudovariety \pv V, the following equality holds, where
two-sided Karnofsky--Rhodes expansion needs to be extended to
profinite semigroups (as in \cite{Rhodes&Steinberg:2001}) when \pv V
is not locally finite:
\begin{displaymath}
  (\Om AV)^{\krho^\omega} = \Om A{(LI\malcev V)}.
\end{displaymath}

\begin{Prop}\label{p:infinite-KRexpansion-is-equidivisible}
  Let $S$ be a finite semigroup generated by $\varphi:A\to S$, where
  $A$ is a finite alphabet. The profinite semigroup
  $S_\varphi^{\krho^\omega}$ is a strong KR-cover with respect to the
  generating mapping $\varphi^{\krho^\omega}:A\to
  S_\varphi^{\krho^\omega}$.
\end{Prop}

\begin{proof}
  Let $\kappa=\varphi^{\krho^\omega}$ and $\psi$ be a continuous
  homomorphism from $S_\varphi^{\krho^\omega}$ onto a finite semigroup
  $T$. Choose an integer $n\geq 1$ for which $\psi$ has a
  factorization $\psi=\psi_n\circ\varrho_n$ for some homomorphism
  $\psi_n:S_\varphi^{\krho^n}\to T$ (its existence being guaranteed,
  for instance, by \cite[Lemma~3.1.37]{Rhodes&Steinberg:2009qt}).
  Observe that the non-dashed part of
  Diagram~\eqref{eq:infinite-KRexpansion-is-equidivisible-1} is
  commutative.
    \begin{equation}\label{eq:infinite-KRexpansion-is-equidivisible-1}
   \begin{split}
    \xymatrix@R=2mm@C=12mm{
      S_\varphi^{\krho^n}
      \ar[dddd]_{\psi_n}
      &&&
      S_\varphi^{\krho^{n+1}}
      \ar[lll]_{\varrho_{n+1,n}}
      \ar@{-->}[dddd]^{(\psi_n)^\krho}
      \\
    &&&\\
    &S_\varphi^{\krho^\omega}
    \ar[ldd]_\psi\ar[luu]^{\varrho_n}
    \ar[rruu]^(0.4){\varrho_{n+1}}
    &
    A^+\ar[l]^\kappa
    \ar[ruu]_{\varphi^{\krho^{n+1}}}
    \ar[rdd]^(.35){(\psi\circ\kappa)^{\krho}}
    &&\\
    &&&\\
    T&&&T_{\psi\circ\kappa}^\krho\ar[lll]_{\pi_{\psi\circ\kappa}}
  }
  \end{split}
  \end{equation}
  
  Applying Proposition~\ref{p:mario-branco-proposition}, we see that
  there is a unique homomorphism
  $(\psi_n)^\krho:S_\varphi^{\krho^{n+1}}\to
  T_{\psi\circ\kappa}^\krho$ for which
  Diagram~\eqref{eq:application-of-mario-branco-proposition-1}
  commutes.
  Since the topmost triangle in
  Diagram~\eqref{eq:application-of-mario-branco-proposition-1} is
  the rightmost triangle in
  Diagram~\eqref{eq:infinite-KRexpansion-is-equidivisible-1}, we
  deduce that the latter is commutative.
    \begin{equation}\label{eq:application-of-mario-branco-proposition-1}
  \begin{split}  
    \xymatrix@C=15mm@R=10mm{
      **[r]{S_\varphi^{\krho^{n+1}}}\ar[dd]_{\varrho_{n+1,n}}\ar@{-->}[rr]^{(\psi_n)^\krho}
      &
      &**[r]{T_{\psi\circ\kappa}^\krho}\ar[dd]^{\pi_{\psi\circ\kappa}}
      \\
      &A^+\ar[dl]_{\varphi^{\krho^n}}\ar[ul]_{\varphi^{\krho^{n+1}}}
      \ar[dr]^{\psi\circ\kappa}\ar[ur]^{(\psi\circ\kappa)^\krho}
      &
      \\
      **[r]{S_\varphi^{\krho^{n}}}\ar[rr]^{\psi_n}
      &
      &T
    }
  \end{split}    
  \end{equation}

  Note that $(\psi_n)^\krho$ is onto, as $(\psi\circ\kappa)^\krho$ is onto.
  Let $\psi_\kappa$
  be the onto continuous homomorphism 
  $(\psi_n)^\krho\circ\varrho_{n+1}$.
  The commutativity of Diagram~\eqref{eq:infinite-KRexpansion-is-equidivisible-1} entails in particular that
  \begin{equation*}
    (\psi\circ\kappa)^\krho=(\psi_n)^\krho\circ\varrho_{n+1}\circ\kappa
    =\psi_\kappa\circ\kappa
  \end{equation*}
  and so Diagram~\eqref{eq:infinite-KRexpansion-is-equidivisible-3}
  commutes.
\begin{equation}\label{eq:infinite-KRexpansion-is-equidivisible-3}
\begin{split}
  \xymatrix@C=12mm{
    A^+\ar[d]_{(\psi\circ\kappa)^\krho}\ar[r]^\kappa&**[r]{S_\varphi^{\krho^\omega}}
    \ar@{-->}[ld]_{\psi_\kappa}
    \ar[d]^{\psi}\\
    **[r]{T_{\psi\circ\kappa}^\krho}\ar[r]^(0.6){\pi_{\psi\circ\kappa}}&T
  }
\end{split}
\end{equation}
This establishes that $S_\varphi^{\krho^\omega}$
is a strong KR-cover.  
\end{proof}

\begin{Examp}
  \label{eg:Sl}
  Let $S=\{a,b\}$ be the 2-element semilattice, where $b$ is the
  minimum element. For the alphabet $A=\{a,b\}$, the description of
  $S$ determines an onto homomorphism $\varphi:A^+\to S$. By
  Proposition~\ref{p:infinite-KRexpansion-is-equidivisible}, the
  profinite semigroup $T=S_\varphi^{\krho^\omega}$ is a strong
  KR-cover with respect to the natural generating mapping
  $\varphi^{\krho^\omega}:\{a,b\}\to T$. Since
  $\varphi^{\krho^\omega}(a)\ne\varphi^{\krho^\omega}(b)$, we see
  $\{a,b\}$ as a subset of $T$.
  %, and $\varphi^{\krho^\omega}$ as the
  % corresponding inclusion.
\end{Examp}

In \cite{Almeida&ACosta&Costa&Zeitoun:2019}, some properties of the
equidivisible profinite semigroups which are letter
supper-cancellative were studied. But the question of whether such
semigroups may not be relatively free was left open (cf.~\cite[Section
9]{Almeida&ACosta&Costa&Zeitoun:2019}). Example~\ref{eg:Sl} provides
an answer to that question, as shown next.

\begin{Prop}
    The semigroup $T$
    from Example~\eqref{eg:Sl} is an equidivisible profinite semigroup
    which is letter super-cancellative but not
    a relatively free profinite semigroup.
  \end{Prop}

  \begin{proof}
    We already observed that $T$ is a strong KR-cover. It then follows
    from Theorem~\ref{t:characterization-StrongKR} that $T$ is
    equidivisible and letter super-cancellative. We proceed to show
    that $T$ is not a relatively free profinite semigroup.

    We claim that the idempotent $\varphi^{\krho^\omega}(b^\omega)\in
    T$ belongs to the minimum ideal of~$T$. To avoid overloaded
    notation, denote $\varphi^{\krho^n}(v)$ by $[v]_n$. We establish
    the claim by proving that the idempotent $[b^\omega]_n$ belongs to
    the minimum ideal of~$T_n=S_\varphi^{\krho^n}$, for every $n\geq
    0$. We do this by showing inductively on~$n$ that %
  \begin{equation}
    \label{eq:eg-Sl}
    \text{$[b^\omega wb^\omega]_n=[b^\omega]_n$ %
      \,
    for every $w\in A^*$.}
\end{equation}
The base case $n=0$ is immediate, as $\varphi^{\krho^0}=\varphi$.
Assume that \eqref{eq:eg-Sl} holds for a certain value of $n\geq 0$. 
  Let $k$ be an integer such that
  $[b^\omega]_{n+1}=[b^{k}]_{n+1}$.
  Note that, since $\varphi^{\krho^{n}}=\varrho_{n+1,n}\circ\varphi^{\krho^{n+1}}$, we also have
  $[b^\omega]_{n}=[b^{k}]_{n}$.
  In the graph $\Gamma=\Gamma_{\varphi^{\krho^n}}$, the path $p_{b^{2k}}$ is
  the concatenation of the path $q$ from $(I,[b^k]_n)$ to
  $([b^k]_n,[b^k]_n)$ labeled by $b^k$, and the path $q'$
  from the latter vertex to
  $([b^k]_n,I)$, which is also labeled by $b^k$.
  On the other hand, by the induction hypothesis, we
  know that, for every $w\in A^*$, the path
  $p_{b^{2k} w b^{2k}}$ of $\Gamma$ decomposes as $qrq'$
  where $r$ is the circuit at the vertex
  $([b^{k}]_n,[b^k]_n)$ labeled by $b^kwb^k$ (see Figure~\ref{fig:paths-in-the-example}).
  In particular, the two coterminal paths $p_{b^{2k}}$ and
  $p_{b^{2k} wb^{2k}}$ use the same transition edges of
  the graph $\Gamma$, and so the equality $[b^{2k}]_{n+1}=[b^{2k}wb^{2k}]_{n+1}$ holds. That is, we have \eqref{eq:eg-Sl} for
  $n+1$ in the place of~$n$. This concludes the inductive step of the proof,
  and shows that $b^\omega$ belongs to the minimum ideal of $T$.

\tikzset{every loop/.style={min distance=10mm,in=240,out=-60,looseness=5}}  

\begin{figure}[h]
  \centering
       \begin{tikzpicture}[shorten >=1pt, node distance=2.5cm and 4cm, on grid,initial text=,semithick,line width=0.5pt]
       %\tikzstyle{accepting}=[accepting by arrow]
 \tikzstyle{state}=[rectangle,draw,minimum size=17pt,inner sep=3pt]
  \tikzset{every loop/.style={min distance=10mm,out=120,in=50,looseness=5}}  
 
  \node[state]   (1)   {$(I,[b^k]_{n})$};
  \node[state]   (2) [right=of 1]  {$([b^k]_{n},[b^k]_{n})$};
  \node[state]   (3) [right=of 2]  {$([b^k]_{n},I)$};
  \path[->]  (1)   edge  [bend left=20] node [above] {$b^k$} (2)
             (2)   edge  [bend left=20] node [above] {$b^k$} (3)
             (2)   edge [loop above]   node         {$b^kwb^k$} ();
\end{tikzpicture}
  
\caption{The path $p_{b^{2k} w b^{2k}}$
  of the graph $\Gamma_{\varphi^{\krho^n}}$.}
  \label{fig:paths-in-the-example}
\end{figure}
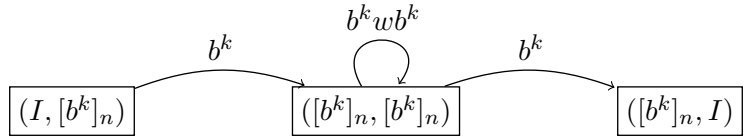

Arguing by contradiction, suppose that $T$ is a relatively free
profinite semigroup $\Om XV$, for some semigroup pseudovariety $\pv V$
and alphabet~$X$. As the semilattice $S=\{a,b\}$ is a continuous
homomorphic image of $T$, the set~$X$ has at least two elements and
$\pv V$ contains the pseudovariety $\pv {Sl}$ (of all finite
semilattices). Let $\pi$ be the unique continuous homomorphism $\Om
XV\to \Om X{Sl}$ mapping each generator to itself. Then
$\pi(b)=\pi(b^\omega)$ belongs to the minimum ideal $K$ of $\Om X{Sl}$
by the claim, and since $\{a,b\}$ generates $T=\Om XV$, we conclude
that $\Om X{Sl}\setminus K\subseteq \pi(a^*)=\{\pi(a)\}$. But since
$X$ has at least two elements, which belong to the set $\Om
X{Sl}\setminus K$,
  we reach a contradiction.
  To avoid the contradiction, the only possibility is that $T$ is not a relatively free profinite semigroup.
  \end{proof}

\section{Profinite coproducts of letter super-cancellative equidivisible profinite semigroups}

The proof of the following proposition is an adaptation
of the proof of Theorem~\ref{t:coproduct-KR-covers}.

\begin{Prop}
  \label{p:coproduct-StrongKR-trivial-covers}
  If \pv V is a semigroup pseudovariety containing \pv{LI},
  then the class of all pro-$\pv V$
  strong KR-covers of the trivial
  semigroup is closed under finite \pv V-coproducts.
\end{Prop}

\begin{proof}
  We start by observing that, by
  Remark~\ref{r:properties-of-the-projection}, \pv V contains all
  two-sided Karnofsky--Rhodes expansions of the trivial semigroup.
  
  Let $(S_i)_{i\in I}$ be a finite family of pro-\pv V
  strong KR-covers of the trivial semigroup $T$.
  Let $S$ be their \pv V-coproduct,
  with associated continuous
  homomorphisms $\varphi_i:S_i\to S$.
  Consider the unique mapping
  $\psi:S\to T$ and let
  $\psi_i=\psi\circ\varphi_i$.
  Since $S_i$ is a
  strong KR-cover of $T$,
  we know that there are a finite set $A_i$,
  a generating mapping $\kappa_i:A_i\to S_i$,
  and a continuous homomorphism $\beta_i$ such that
  the diagram
  \begin{equation*}
    \xymatrix@C=14mm@R=11mm{
      A_i^+\ar[r]^{\kappa_i}
      \ar[d]_(.4){(\psi_i\circ\kappa_i)^\krho}
      &S_i\ar[d]^(.4){\psi_i}\ar@{-->}[dl]_{\beta_i}
      \\
      **[r]{{T}_{\psi_i\circ\kappa_i}^\krho}
      \ar[r]_(.6){\pi_{\psi_i\circ\kappa_i}}
      &T
    }
  \end{equation*}
  commutes.
  We may assume that
  the sets $A_i$
  are pairwise disjoint.
  Let
  $A=\bigcup_{i\in I}A_i$.
  The union $\kappa=\bigcup_{i\in I}\kappa_i$
  is then a generating mapping $A\to S$
  with finite domain.
  
  Applying Proposition~\ref{p:mario-branco-proposition},
  we see that for each $i\in I$
  there is a homomorphism 
  $\eta_i:T_{\psi_i\circ\kappa_i}^\krho\to T_{\psi\circ\kappa}^\krho$
  such that Diagram~\eqref{eq:eta-i}
  commutes.
  \begin{equation}\label{eq:eta-i}
  \begin{split}  
    \xymatrix@C=14mm@R=10mm{
      **[r]{T_{\psi_i\circ\kappa_i}^\krho}\ar[dd]_{\pi_{\psi_i\circ\kappa_i}}\ar@{-->}[rrr]^{\eta_i}
      &
      &
      &**[r]{T_{\psi\circ\kappa}^\krho}\ar[dd]^{\pi_{\psi\circ\kappa}}
      \\
      &A_i^+\ar[dl]_{\psi_i\circ\kappa_i}\ar@{^{((}->}[r]\ar[ul]_(0.4){{\psi_i\circ\kappa_i}^\krho}
      &A^+\ar[dr]^{\psi\circ\kappa}\ar[ur]^(0.4){(\psi\circ\kappa)^\krho}
      &
      \\
      T\ar@{=}[rrr]^{\mathrm{Id}_{T}}
      &
      &
      &T
    }
  \end{split}    
  \end{equation}
  Therefore, the non-dashed part
  of Diagram~\eqref{eq:StrongKR-trivial-slicing-T} commutes.
  \begin{equation}
  \begin{split}
    \label{eq:StrongKR-trivial-slicing-T}
    \xymatrix@C=15mm@R=8mm{
      A_i^+
      \ar[rd]^{\kappa_i}
      \ar[ddd]_{(\psi_i\circ\kappa_i)^\krho}
      \ar@{^{((}->}[rrr]
      &
      &
      &
      A^+
      \ar[ddd]^{(\psi\circ\kappa)^\krho}
      \ar[ld]_{\kappa}
      \\
      & S_i
      \ar[ldd]_(.4){\beta_i}
      \ar[d]^{\psi_i}
      \ar[r]^{\varphi_i}
      & S
      \ar[dl]^\psi
      \ar@{-->}[ddr]^{\beta}&\\
      & T
      &&
      \\
      **[r]{{T}_{\psi_i\circ\kappa_i}^\krho}
      \ar[rrr]^{\eta_i}
      \ar[ru]_(.6){\pi_{\psi_i\circ\kappa_i}}
      &
      &
      &
      **[r]{T_{\psi\circ\kappa}^\krho}
      \ar[ull]^{\pi_{\psi\circ\kappa}}
    }
  \end{split}  
\end{equation}
From the observation at the beginning of the proof, we know that
$T_{\psi\circ\kappa}^\krho$ belongs to~\pv V. Therefore, by the
definition of $\pv V$-coproduct, there is a unique continuous
homomorphism $\beta:S\to T_{\psi\circ\kappa}^\krho$ such that
$\beta\circ\varphi_i=\eta_i\circ\beta_i$ for each~$i\in I$, thus
completing Diagram~\eqref{eq:StrongKR-trivial-slicing-T}, which we
next show to be commutative. Indeed, using the already established
commutativity of the non-dashed part of
Diagram~\eqref{eq:StrongKR-trivial-slicing-T}, we see that for every
$i\in I$ and every $a\in A_i$, the following chain of equalities
holds:
  \begin{equation*}
    \beta\circ\kappa(a)=\beta\circ\varphi_i\circ\kappa_i(a)=
    \eta_i\circ\beta_i\circ\kappa_i(a)
    =\eta_i\circ(\psi_i\circ\kappa_i)^\krho(a)
    =(\psi\circ\kappa)^\krho(a).
  \end{equation*}
  This shows that $\beta\circ\kappa=(\psi\circ\kappa)^\krho$.
  In particular, we conclude that Diagram~\eqref{eq:conclusion-of-the-proof}
  \begin{equation}\label{eq:conclusion-of-the-proof}
    \begin{split}
      \xymatrix@C=14mm@R=11mm{
        A^+
        \ar[r]^{\kappa}
        \ar[d]_(.4){(\psi\circ\kappa)^\krho}
        &
        S
        \ar[d]^(.4){\psi}
        \ar@{-->}[dl]_{\beta}
        \\
        **[r]{{T}_{\psi\circ\kappa}^\krho}
        \ar[r]_(.6){\pi_{\psi\circ\kappa}}
        &T }
    \end{split}
  \end{equation}
  commutes, thereby completing the proof that $S$ is a
  strong KR-cover of the trivial semigroup~$T$.
\end{proof}

We are ready to deduce the following theorem.

\begin{Thm}\label{t:coproduct-StrongKR}
  For every pseudovariety of semigroups \pv V closed under two-sided
  Karnofsky--Rhodes expansion, the class of letter super-cancellative
  equidivisible finitely generated pro-$\pv V$ semigroups is closed
  under finite \pv V-coproducts.
\end{Thm}

\begin{proof}
  Let $\mathcal C$ be the class of all pro-\pv V KR-covers, and let
  $\mathcal D$ be the class of all pro-\pv V strong KR-covers of the
  trivial semigroup. By
  Theorem~\ref{t:characterization-of-pseudovarieties-closed-under-KR-expansion},
  \pv V contains~\pv{LI}.\footnote{In fact, we only need that \pv V
    contains~\pv{LI} to invoke
    Proposition~\ref{p:coproduct-StrongKR-trivial-covers}, where the
    proof starts by observing that it follows that \pv V contains all
    two-sided Karnofsky--Rhodes expansions of the trivial semigroup.
    So, one could avoid applying
    Theorem~\ref{t:characterization-of-pseudovarieties-closed-under-KR-expansion}
    here. We preferred to keep the statement of
    Proposition~\ref{p:coproduct-StrongKR-trivial-covers} the simplest
    possible.}
  Both $\mathcal C$ and $\mathcal D$ are closed under
  taking finite \pv V-coproducts, by
  Theorem~\ref{t:coproduct-KR-covers} and
  Proposition~\ref{p:coproduct-StrongKR-trivial-covers}, respectively.
  By Theorem~\ref{t:characterization-StrongKR}, the class of letter
  super-cancellative equidivisible pro-$\pv V$ semigroups is the
  intersection $\mathcal C\cap\mathcal D$. Combining these
  observations, we immediately obtain the theorem.
\end{proof}

A natural question arising from Theorem~\ref{t:coproduct-StrongKR}
is the following problem, which we leave open.

\begin{Problem}
  Suppose that $\pv V$ is a pseudovariety of semigroups closed under
  two-sided Karnofsky--Rhodes expansion. Is it true that the class of
  all equidivisible pro-\pv V semigroups is closed under \pv
  V-coproducts? A perhaps simpler question is the following: is the
  class of all finitely generated equidivisible pro-\pv V semigroups
  closed under finite \pv V-coproducts?
\end{Problem}

It would also be interesting to have a complete characterization 
of the KR-covers, in the spirit  of the characterization
of the strong KR-covers, established in
Theorem~\ref{t:characterization-StrongKR}.

\end{document}